\UseRawInputEncoding
\documentclass[12pt, reqno]{amsart}
\usepackage[margin=1in]{geometry}
\usepackage{amssymb,latexsym,amsmath,amscd,amsfonts}
\usepackage{latexsym}
\usepackage[mathscr]{eucal}
\usepackage{bm}
\usepackage{mathptmx}
\usepackage{xcolor}
\usepackage{amssymb}
\usepackage{amsthm}
\usepackage{dcolumn}
\usepackage[all]{xy}
\usepackage{enumitem}
\usepackage[utf8]{inputenc}

\def \qed {\hfill \vrule height6pt width 6pt depth 0pt}
\def\textmatrix#1&#2\\#3&#4\\{\bigl({#1 \atop #3}\ {#2 \atop #4}\bigr)}
\def\dispmatrix#1&#2\\#3&#4\\{\left({#1 \atop #3}\ {#2 \atop #4}\right)}
\newcommand{\beg}{\begin{equation}}
	\newcommand{\eeg}{\end{equation}}
\newcommand{\ben}{\begin{eqnarray*}}
	\newcommand{\een}{\end{eqnarray*}}

\newtheorem{thm}{Theorem}[section]
\newtheorem{cor}[thm]{Corollary}
\newtheorem{lem}[thm]{Lemma}

\newtheorem{prop}[thm]{Proposition}
\numberwithin{equation}{section} \theoremstyle{definition}
\newtheorem{defn}[thm]{Definition}

\newtheorem{eg}[thm]{Example}

\newcommand{\C}{\mathbb{C}}

\newcommand{\G}{\mathbb{G}_2}
\newcommand{\D}{\mathbb{D}}
\newcommand{\T}{\mathbb{T}}
\newcommand{\N}{\mathbb{N}}

\newcommand{\HS}{\mathcal{H}}

\newcommand{\DC}{\overline{\mathbb{D}}}

\def\textmatrix#1&#2\\#3&#4\\{\bigl({#1 \atop #3}\ {#2 \atop #4}\bigr)}
\def\dispmatrix#1&#2\\#3&#4\\{\left({#1 \atop #3}\ {#2 \atop #4}\right)}

\begin{document}
	
	\title[Toral contractions and $\Gamma$-distinguished $\Gamma$-contractions]{The toral contractions and $\Gamma$-distinguished $\Gamma$-contractions}
	
	\author[Pal and Tomar]{SOURAV PAL AND NITIN TOMAR}
	
	\address[Sourav Pal]{Mathematics Department, Indian Institute of Technology Bombay, Powai, Mumbai-400076, India.} \email{sourav@math.iitb.ac.in, souravmaths@gmail.com}	
	
	\address[Nitin Tomar]{Mathematics Department, Indian Institute of Technology Bombay, Powai, Mumbai-400076, India.} \email{tnitin@math.iitb.ac.in, tomarnitin414@gmail.com}		
	
	
	
	\thanks{The first named author is supported in part by the ``Core Research
Grant'' of Anusandhan National Research Foundation (ANRF), Govt. of India, with
Grant No. CRG/2023/005223. The second named author was supported by
the Prime Minister's Research Fellowship of Govt. of India with award number PMRF-1300140 during the course of the paper. At present, the second named author is supported by a fellowship from the first named author's `Core Research Grant (CRG)' of ANRF with Award No. CRG/2023/005223.}	
	
	\begin{abstract}
A pair of commuting Hilbert space contractions $(T_1,T_2)$ is said to be \textit{toral} if there is a polynomial $p \in \mathbb C[z_1,z_2]$ such that its zero set $Z(p)$ defines a distinguished variety in the bidisc $\mathbb D^2$ and $p(T_1,T_2)=0$. A pair of commuting Hilbert space operators $(S,P)$ is said to be a $\Gamma$-\textit{contraction} if the closed symmetrized bidisc
\[
\Gamma=\{ (z_1+z_2,z_1z_2)\,:\, |z_1|, \, |z_2| \leq 1 \}
\]
is a spectral set for $(S,P)$. A $\Gamma$-contraction $(S,P)$ is called $\Gamma$-\textit{distinguished} if $q(S,P)=0$ for some polynomial $q\in \mathbb C[z_1,z_2]$ whose zero set $Z(q)$ gives rise to a distinguished variety in the symmetrized bidisc $\mathbb G_2$. We find necessary and sufficient conditions such that a toral pair of contractions dilates to a toral pair of isometries. In the same spirit, we characterize all $\Gamma$-distinguished $\Gamma$-contractions that admit dilation to $\Gamma$-distinguished $\Gamma$-isometries. The distinguished boundary of a distinguished variety in $\mathbb D^2$ and $\mathbb G_2$ is determined. Examples are provided at places to show the contrasts between the theory of toral contractions and $\Gamma$-distinguished $\Gamma$-contractions.
	\end{abstract}

	\maketitle
	
	\section{Introduction} 
	
	\vspace{0.1cm}
	
\noindent Throughout the paper, all operators are bounded linear operators acting on complex Hilbert spaces. A contraction is an operator with norm not greater than $1$. The space of all operators acting on a Hilbert space $\HS$ is denoted by $\mathcal{B}(\HS)$. The following notations will be used:  $\C$ denotes the complex plane, $\D$ stands for the open unit disk $\{z: |z|<1\}$, $\T$ is the unit circle $\{z : |z|=1\}$ and $\mathbb{E}=\C \setminus \DC$, i.e., the complement of $\DC$. Spectral set, complete spectral set, distinguished boundary and rational dilation are defined in Section \ref{sec02}.
   
\smallskip

A pair of commuting contractions $(T_1,T_2)$ is said to be \textit{toral} if $p(T_1,T_2)=0$ for a polynomial $p$ in $\C[z_1,z_2]$ such that the zero set $Z(p)$ of $p$ intersects the bidisc $\D^2$ and exits through the distinguished boundary of $\DC^2$, the $2$-torus $\T^2$ without intersecting any other part of its topological boundary $\partial \DC^2$, that is, $Z(p)\cap \D^2 \neq \emptyset$ and $Z(p)\cap \partial \DC^2 = Z(p)\cap \T^2$. In the literature \cite{AglerMcCarthy}, such a set $Z(p)\cap \D^2$ is called a \textit{distinguished variety in} $\D^2$. Also, a polynomial $p$ in $\C[z_1,z_2]$ for which $Z(p)\cap \D^2$ is a distinguished variety in $\D^2$ is called a \textit{toral polynomial}. Ando's celebrated theorem \cite{Ando} tells us that every commuting pair of contractions dilates to a commuting pair of isometries. Thus, it is naturally asked if a toral pair of contractions dilates to a toral pair of isometries. Indeed, it remains an open problem for quite sometime, e.g. see \cite{DasII}. In this article, we address this problem by finding necessary and sufficient conditions and thus our first main result is the following.

\begin{thm}\label{main_toral}
		Let $(T_1, T_2)$ be a toral pair of commuting contractions acting on a Hilbert space $\HS$. Then the following are equivalent.
		\begin{enumerate}		
		\item $(T_1, T_2)$ dilates to a toral pair of commuting isometries.
		
        \item There is a Hilbert space $\mathcal{K} \supseteq \HS$, a toral pair of isometries $(D_1, D_2)$ on $\HS^\perp=\mathcal{K}\ominus \HS$ and $C_1, C_2 \in \mathcal{B}(\HS, \HS^\perp)$ such that the following hold:
		\begin{itemize}
			\item[$(i)$] $C_1T_2+D_1C_2=C_2T_1+D_2C_1$
			
			\item[$(ii)$] $C_1^*D_1=C_2^*D_2=0$
			
			\item[$(iii)$] $C_i^*C_i=D_{T_i}^2$ \quad  for  \ $i=1, 2 \,$.			
			\end{itemize}
			
			\item There is a toral polynomial $q(z_1, z_2)$ such that $Z(q) \cap \DC^2$ is a complete spectral set for $(T_1, T_2)$.		
		\end{enumerate}
		Moreover, if $f$ and $g$ are toral polynomials such that $f(T_1, T_2)=g(D_1, D_2)=0$, then $Z(fg) \cap \DC^2$ is a complete spectral set for $(T_1, T_2)$.
	\end{thm}
	
We shall prove this theorem in Section \ref{sec07}. In Section \ref{sec06}, we invoke the theory of analytic variety in a domain (e.g. see \cite{Rossi, Stanislaw}) to show that the distinguished boundary of a distinguished variety $Z(p)\cap \D^2$ is equal to $Z(p)\cap \T^2$ and by an application of this we prove in Theorem \ref{Arveson_toral} that a pair of commuting contractions $(T_1,T_2)$ dilates to a toral pair of unitaries if and only if there is a toral polynomial $q$ so that $Z(q) \cap \DC^2$ is a complete spectral set for $(T_1, T_2)$. Also, we obtain characterizations for a toral pair of commuting isometries in Theorem \ref{toral_isometries}. In fact Section \ref{sec06} sets the platform for proving the $(1) \Leftrightarrow (3)$ part of Theorem \ref{main_toral}.
	
\smallskip	

	 We intend to establish analogous results for the symmetrized bidisc. Indeed, this paper is more of a study of the symmetrized bidisc than the bidisc. The symmetrized bidisc $\mathbb G_2$ is a non-convex but polynomially convex domain in $\C^2$ defined by
	\[
	\mathbb{G}_2:=\{(tr(A), det(A)) \ : \ A=[a_{ij}]_{2 \times 2},  \|A\| <1  \} \subseteq \mathbb{C}^2.
	\]	
	The symmetrized bidisc has its origin in $2 \times 2$ spectral Nevanlinna-Pick interpolation. The general $n \times n$ spectral Nevanlinna-Pick interpolation problem states the following: given distinct points $z_1, \dotsc, z_n$ in $\D$ and $n \times n$ matrices $F_1, \dotsc, F_n$ in the spectral unit ball $\Omega_n$ of the space of $n \times n$ matrices $\mathcal{M}_n(\C)$, whether or under what conditions it is possible to find an analytic function $f: \D \to \Omega_n$ such that $f(z_i)=F_i, \ i=1, \dotsc, n$. It is obvious that a $2 \times 2$ matrix $F$ is in $\Omega_2$ if and only if its eigenvalues $\mu_1, \mu_2$ are in $\D$ and this happens if and only if $(\text{tr}(F), \text{det}(F))=(\mu_1+\mu_2, \mu_1\mu_2)$ belongs to $\G$. The symmetrization map $\pi: \C^2 \to \C^2$ is defined by $\pi(z_1, z_2)=(z_1+z_2, z_1z_2)$. The {symmetrized bidisc} $\G$ and its closure $\Gamma$, turn out to be the images of the bidisc $\D^2$ and its closure $\overline{\D}^2$ respectively under $\pi$, that is,
	\begin{align*}
		&\G=\pi(\D^2)=\{(z_1+z_2, z_1z_2) \ : \ |z_1|<1, |z_2|<1 \},\\
		&\Gamma= \overline{\mathbb G}_2=\pi(\DC^2)=\{(z_1+z_2, z_1z_2) \ : \ |z_1|\leq 1, |z_2| \leq 1 \}.
	\end{align*}
	The main motivation behind studying the symmetrized bidisc is that the $2 \times 2$ Nevanlinna-Pick interpolation problem reduces to a similar interpolation problem of $\G$ in the following way.

	\begin{prop}[\cite{Nikolov}, Proposition 1.1]
		Let $\alpha_1, \dotsc, \alpha_n$ be distinct points in $\D$ and let $F_1=\lambda_1I, \dotsc, \\ F_k=\lambda_kI \in \Omega_2$ be scalar matrices. Also let $F_{k+1}, \dotsc, F_n \in \Omega_2$ be non-scalar matrices. Suppose $\phi=(\phi_1, \phi_2) : \D \to \G$ is a holomorphic map such that $\phi(\alpha_j)=\sigma(F_j)=(\text{tr}(F_j), \text{det}(F_j))$ for $j=1, \dotsc, n$. Then there exists a holomorphic map $\psi:\D \to \Omega_2$ satisfying $\phi=\sigma\circ \psi$ and $\psi(\alpha_j)=F_j$ for $j=1, \dotsc, n$ if and only if $\phi_2'(\alpha_j)=\lambda_j\phi_1'(\alpha_j)$ for $j=1, \dotsc, k$.
	\end{prop}
		
 Obviously a bounded domain like $\G$ that has complex-dimension $2$, is much easier to deal with than a norm-unbounded object like $\Omega_2$ which has complex-dimension $4$. On the other hand, $\G$ was the first example of a non-convex domain in which the Caratheodory and Kobayashi distances coincide (see \cite{AglerIV}). The symmetrized bidisc has a rich literature from complex geometric, function theoretic and operator theoretic points of view, e.g. see \cite{AglerI15, AglerII16, AglerIV, Agler-Young-V, Pal8, Bhatt-Pal, PalShalit1} and the references therein.
 
 \medskip
	
It follows from Ando's theorem \cite{Ando} that a pair of commuting operators $T_1,T_2$ are contractions if and only if $\DC^2$ is a spectral set for $(T_1,T_2)$. This leads to considering a commuting pair of operators having $\Gamma$ as a spectral set.

\begin{defn}
A commuting pair of operators $(S,P)$ acting on a Hilbert space $\HS$ is said to be 

\begin{itemize}
 \item[(i)] a $\Gamma$-\textit{contraction} if $\Gamma$ is a spectral set for $(S,P)$, that is, the Taylor joint spectrum $\sigma_T(S,P) \subseteq \Gamma$ and von Neumann's inequality
\[
\|f(S,P)\| \leq \sup_{(s,p) \in \Gamma} \; |f(s,p)| = \|f\|_{\infty, \Gamma}
\]
holds for all rational functions $f=p \slash q$ with $p,q \in \C[z_1,z_2]$ and $q$ having no zeros in $\Gamma$ ;

\item[(ii)] a $\Gamma$-\textit{unitary} if $S,P$ are normal operators and $\sigma_T(S,P) \subseteq b\Gamma$ ;

\item[(iii)] a $\Gamma$-\textit{isometry} if there is a Hilbert space $\widetilde{\HS} \supseteq \HS$ and a $\Gamma$-unitary $(\widetilde{S}, \widetilde{P})$ acting on $\widetilde{\HS}$ such that $\HS$ is a common invariant subspace for $S,P$ and that $S=\widetilde{S}|_{\HS}, \; P=\widetilde P|_{\HS}$;

\item[(iv)] a \textit{pure} $\Gamma$-\textit{contraction} if $(S,P)$ is a $\Gamma$-contraction and $P$ is a pure contraction, i.e., ${P^*}^n \rightarrow 0$ strongly as $n \rightarrow \infty $.
 
\end{itemize} 

\end{defn}

\noindent The primary objects of study in this article are distinguished variety, polynomial defining a distinguished variety in $\D^2$ and $\mathbb G_2$ and operator pairs annihilated by polynomials. A commuting operator pair $(T_1,T_2)$ that is annihilated by a polynomial $q \in \C[z_1,z_2]$, i.e., satisfying $q(T_1,T_2)=0$ is called an \textit{algebraic pair}. Below we define a distinguished variety in a domain in $\C^n$.
\begin{defn}
Given a bounded domain $\Omega$ in $\C^n$, a nonempty set $V \subseteq \Omega$ is said to be a \textit{distinguished variety} in $\Omega$ if there is an algebraic variety $W \subset \C^n$ such that $V=W \cap \Omega $ and $W \cap \partial \overline{\Omega}= W \cap b\overline{\Omega}$, where $b\overline{\Omega}$ is the distinguished boundary of $\overline{\Omega}$.	
\end{defn}
The distinguished varieties in domains like bidisc, symmetrized bidisc, tetrablock or even more generally in polydisc and symmetrized polydisc have been extensively studied in past two decades, e.g. see \cite{AglerMcCarthy, AglerKneseMcCarthy2, Bhatt-Sau-Kumar, Das, DasII, Dritschel, Knese9, PalShalit1, PalS12, PalS, PalS11}. One of the most pioneering works in operator theory is And\^{o}'s inequality \cite{Ando}, which states that if $(T_1, T_2)$ is a commuting pair of contractions, then for every polynomial $p \in \C[z_1, z_2]$,
 \[
 \|p(T_1, T_2)\|\leq \sup\{|p(z_1, z_2)| : (z_1, z_2) \in \DC^2 \}.
 \] 
 The main motivation behind studying distinguished varieties is the following improvement of Ando's inequality by Agler and McCarthy. 
%--------------------------------------------------------------------------------------
\begin{thm}[Agler and McCarthy, \cite{AglerMcCarthy}]
	Let $T_1, T_2$ be two commuting contractive matrices, neither of which has an eigenvalue of unit modulus. Then  there is a distinguished variety $V$ in $\D^2$ such that 
	\[
	\|p(T_1, T_2)\| \leq \sup_{(z_1,z_2) \in V} \,\, |p(z_1, z_2)| \quad \text{ for every } \, \, p \in \C[z_1, z_2].
	\]
	
\end{thm}
%----------------------------------------------------------------------------------- 
A similar result for the symmetrized bidisc was proved by Pal and Shalit in \cite{PalShalit1}. Also, an explicit description of all distinguished varieties in the bidisc and the symmetrized bidisc were given in \cite{AglerMcCarthy} and \cite{PalShalit1} respectively. To go parallel with the bidisc, we shall use the following terminologies for the symmetrized bidisc.

\begin{defn}
A polynomial $p \in \C[z_1, z_2]$ is said to be \textit{$\Gamma$-distinguished} if its zero set $Z(p)$ defines a distinguished variety with respect to $\G$, i.e., $Z(p) \cap \G \ne \emptyset$ and $Z(p) \cap \partial \Gamma =Z(p) \cap b\Gamma$. Also, a $\Gamma$-contraction $(S, P)$ is called \textit{$\Gamma$-distinguished} if it is annihilated by a $\Gamma$-distinguished polynomial, i.e., there is a $\Gamma$-distinguished polynomial $q \in \C[z_1, z_2]$ such that $q(S, P)=0$.		
	\end{defn}
		
Every $\Gamma$-contraction dilates to a $\Gamma$-isometry as was shown in \cite{AglerI15} and \cite{Pal8}. Below we present necessary and sufficient conditions such that a $\Gamma$-distinguished $\Gamma$-contraction dilates to a $\Gamma$-distinguished $\Gamma$-isometry. This is an analogue of Theorem \ref{main_toral} for the symmetrized bidisc and is another main result of this paper.	
	
	\begin{thm}\label{thm814}	
		Let $(S, P)$ be a $\Gamma$-distinguished $\Gamma$-contraction acting on a Hilbert space $\HS$. Then the following are equivalent.
	\begin{enumerate}
	
	\item $(S, P)$ dilates to a $\Gamma$-distinguished $\Gamma$-isometry.

	\item There is a Hilbert space $\mathcal{K}$ containing $\HS$, a $\Gamma$-distinguished $\Gamma$-isometry $(D_1, D_2)$ on $\HS^{\perp}=\mathcal{K}\ominus \HS$ and $C_1, C_2 \in \mathcal{B}(\HS, \HS^\perp)$ such that the following hold:	
	\begin{equation}\label{eqn1002}
			\begin{split}
					& (i) \ \ C_1P+D_1C_2=C_2S+D_2C_1 \ \ (ii)  \ \ S-S^*P=C_1^*C_2  \\ & (iii)  \ \ C_2^*D_2=0  \ \ (iv) \ \ C_2^*C_2=D_P^2 \ \ (v) \ \ C_1=D_1^*C_2.
			\end{split}
		\end{equation}	
	
	\item There is a $\Gamma$-distinguished polynomial $p(z_1, z_2)$ such that $Z(p) \cap \Gamma$ is a complete spectral set for $(S,P)$.
\end{enumerate}		
 Moreover, if $f$ and $g$ are $\Gamma$-distinguished polynomials that annihilate $(S, P)$ and $(D_1, D_2)$ respectively, then $Z(fg) \cap \Gamma$ is a complete spectral set for $(S, P)$. 	
\end{thm}

We prove this theorem in Section \ref{sec08}. In Section \ref{sec06}, we determine the distinguished boundary of a distinguished variety $Z(p) \cap \Gamma$ in the symmetrized bidisc. More precisely, we show that $b(Z(p) \cap \Gamma)= Z(p) \cap b\Gamma$. It is well-known that a $\Gamma$-isometry naturally extends to a $\Gamma$-unitary. We show in Section \ref{sec08} that every $\Gamma$-distinguished $\Gamma$-isometry extends to a $\Gamma$-distinguished $\Gamma$-unitary. A $\Gamma$-contraction $(S, P)$ that dilates to a $\Gamma$-distinguished $\Gamma$-unitary (or $\Gamma$-isometry) $(T, U)$ must be $\Gamma$-distinguished. Indeed, if $p$ is a $\Gamma$-distinguished polynomial such that $p(T, U)=0$, then  $p(S,P)=P_\HS p(T, U)|_\HS=0$. In Example \ref{3.14}, we construct a $\Gamma$-contraction that is not $\Gamma$-distinguished. 

\smallskip

The theory of $\Gamma$-distinguished $\Gamma$-contractions is not a straight-forward analogue of the theory of toral contractions. In fact, these two classes of commuting operators exhibit behaviours that are in stark contrast with each other. For example, Agler, Knese and McCarthy \cite{AglerKneseMcCarthy2} found the following necessary condition for an algebraic pair of commuting pure isometries. 
	
	\begin{thm}[\cite{AglerKneseMcCarthy2}, Theorem 1.8] \label{1.2} For every algebraic pair of commuting pure isometries $(V_1, V_2)$, there is a polynomial $q \in \C[z_1, z_2]$ such that $Z(q) \subseteq \D^2 \cup \T^2 \cup \mathbb{E}^2$ and $q(V_1, V_2)=0$.
	\end{thm}

 Such a polynomial $q$ is referred to as an \textit{inner toral} polynomial in \cite{AglerKneseMcCarthy2}. These polynomials have been further studied in detail by Knese in \cite{Knese9}. Evidently, an inner toral polynomial $q$ satisfies $Z(q) \cap \partial \D^2 =Z(q)\cap \T^2$. Thus, multiplication of an inner toral polynomial $q$ with any toral polynomial $q_0$ makes $qq_0$ a toral polynomial. It then follows from Theorem \ref{1.2} that every algebraic pair of commuting pure isometries is a toral pair. An analogue of this result does not hold for the symmetrized bidisc. In Example \ref{3.14}, we construct an algebraic pure $\Gamma$-isometry that cannot be annihilated by any $\Gamma$-distinguished polynomial. The underlying reason is that a (pure) $\Gamma$-isometry may or may not be the symmetrization of a commuting pair of (pure) isometries. In Section \ref{sec04}, we characterize all $\Gamma$-isometries which arise as the symmetrization of a commuting pair of isometries. 
  
 %\vspace{0.1cm}

	\section{Preliminaries}\label{sec02}
	
	%\vspace{0.3cm}
	
	\noindent 	In this Section, we recall a few basic concepts from the literature. These facts will be used frequently throughout the paper. We begin with the definition of spectral and complete spectral set.		
				
	\subsection{Spectral set and complete spectral set.}
	For a compact subset $X$ of $\C^n$, let
	$Rat(X)$ be the algebra of
	rational functions $p\slash q$, where $p,q \in \C[z_1, \dots , z_n]$ such that $q$
	does not have any zeros in $X$. Let $\underline{T}=(T_1, \dotsc, T_n)$ be a commuting tuple of operators on a Hilbert space $\mathcal{H}$. The set $X \subseteq \mathbb{C}^n$ is said to be a \textit{spectral set} for $\underline{T}$ if the Taylor joint spectrum $\sigma_T(\underline{T})$ of $\underline{T}$ is contained in $X$ and 
	\begin{equation}\label{vNe}
	\|f(\underline{T})\| \leq \|f\|_{\infty, X}=\sup\{|f(\xi)| \ : \ \xi \in X  \} 
	\end{equation}
	for every $f \in Rat(X)$. If (\ref{vNe}) holds for every matricial rational function $F=[f_{ij}]_{m \times m}$, then $X$ is said to be a \textit{complete spectral set} for $\underline{T}$. Note that for a matricial rational function $F=[f_{ij}]_{m \times m}$, where each $f_{ij} \in \, Rat(X)$, we denote by $F(T_1, \dotsc, T_n)$ the block matrix of operators $[f_{ij}(T_1, \dotsc, T_n)]_{m \times m}$ and the right hand side of the inequality in (\ref{vNe}) is the following here: 
	\[
	\|F\|_{\infty, X}=\sup\{\|[f_{ij}(\xi)]_{m\times m}\| \ : \ \xi \in X  \}.
	\]
	
	\subsection{Distinguished boundary and rational dilation.}\label{sub2.2} Let $X\subset \C^n$ be a compact set. A \textit{boundary} for $X$ is a closed subset
	$C$ of $X$ such that every function in $Rat(X)$ attains its maximum modulus
	on $C$. It follows from the theory of uniform algebras that the intersection of all
	the boundaries of $X$ is also a boundary of $X$ and it is the smallest among all
	boundaries. This is called the \textit{distinguished boundary} of $X$ and is
	denoted by $bX$. For a bounded domain $\Omega \subset \C^n$, we denote by $b\Omega$ or $b \overline{\Omega}$, the
	distinguished boundary of $\overline{\Omega}$, and for the sake of simplicity sometimes we call it just the
	distinguished boundary of $\Omega$.
	
	\smallskip	
	
Let $X$ be a spectral set for a commuting $n$-tuple $\underline{T}=(T_1, \dots , T_n)$. Then $\underline T$ is said to have a \textit{rational dilation} or \textit{normal $bX$-dilation} if there exist a Hilbert space $\mathcal{K},$ an isometry $V: \mathcal{H} \to \mathcal{K}$ and a commuting $n$-tuple of normal operators $\underline{N}$ on 
	$\mathcal{K}$ with $\sigma_T(\underline{N}) \subseteq bX$ such that 
	\begin{equation}\label{eq:Rat}
		f(\underline{T})=V^*f(\underline{N})V
	\end{equation}
	for all $f \in Rat(X)$, or simply $f(\underline{T})=P_\mathcal{H}f(\underline{N})|_\mathcal{H}$ for every $f \in Rat(X)$ when $\mathcal{H}$ is realized as a closed linear subspace of $\mathcal{K}$ and $P_\HS$ is the orthogonal projection of $\mathcal{K}$ onto the space $\mathcal{H}$. The dilation is said to be \textit{minimal} when 
	\[
	\mathcal{K}=\overline{\mbox{span}}\{f(\underline{N})h \ :\ h \in \mathcal{H} \ \mbox{ and } \ f \in Rat(X)\}.
	\]

The following two elementary results are well-known and are useful in our context. 

	\begin{lem}
		Let $(T_1, \dotsc, T_n)$ be a commuting tuple of normal operators on a Hilbert space $\mathcal{H}$. Then a compact set $K$ is a spectral set for $(T_1, \dotsc, T_n)$ if and only if $\sigma_T(T_1, \dotsc, T_n)\subseteq K$.
	\end{lem}	
One can find a proof to this result in the literature, e.g. see \cite{TaylorI12} or Lemma 5.8 in \cite{PalS}. Also, the following proposition can be located as Lemma 3.12 in \cite{PalS}.

\begin{prop}\label{basicprop:01} 
		Let $X$ be a compact polynomially convex set in $\C^n$ and let $\underline{T}=(T_1, \dotsc, T_n)$ be a commuting tuple of operators for which $X$ is a spectral set. Then $\underline{T}$ admits a rational dilation if and only if  there exist a Hilbert space $\mathcal{K},$ an isometry $V: \mathcal{H} \to \mathcal{K}$ and a commuting $n$-tuple of normal operators $\underline{N}$ on 
		$\mathcal{K}$ with $\sigma_T(\underline{N}) \subseteq bX$ such that 
		\begin{equation}\label{eq-new1}
			p(\underline{T})=V^*p(\underline{N})V
		\end{equation}
		for every polynomial $p$ in $\C[z_1, \dotsc, z_n]$.	
	\end{prop}
	
	Naturally, our curiosity extends to ask how these classical concepts (i.e., spectral set, complete spectral set and rational dilation) are related with each other. The following famous theorem due to Arveson combines them beautifully.
	
	\begin{thm}[Arveson, \cite{ArvesonI}]\label{thm_Arveson} Let $\underline{T}$ be an $n$-tuple of commuting operators on a Hilbert space $\HS$ for which a compact set $X \subseteq \C^n$ is a spectral set. Then $X$ is a complete spectral set for $\underline{T}$ if and only if $\underline{T}$ has a normal $bX$-dilation.	
	\end{thm}

	\section{The algebraic pairs}\label{sec03}
	
	\vspace{0.1cm}

\noindent Recall that an algebraic pair is a commuting pair of operators $(T_1,T_2)$ that is annihilated by a polynomial $p$ in $\C[z_1,z_2]$. We begin this Section with various notions of algebraic pairs associated with the bidisc and the symmetrized bidisc. However, the main aim of this Section is to study an algebraic pair associated with the symmetrized bidisc, namely the algebraic pure $\Gamma$-isometries. Needless to mention that an algebraic pure $\Gamma$-isometry $(S,P)$ is an algebraic pair and a $\Gamma$-isometry such that $P$ is a pure isometry i.e., an isometry with ${P^*}^n \rightarrow 0$ strongly as $n \rightarrow \infty$. We recall from the `Introduction' that a polynomial $p \in \C[z_1, z_2]$ is said to be toral or $\Gamma$-distinguished if its zero set defines a distinguished variety in the bidisc or the symmetrized bidisc respectively.

\medskip

 \textbf{Note:} We want to emphasize that the term \textit{toral polynomial} is used in two different contexts in the literature. One definition is attributed to Agler et al. \cite{AglerMcCarthyStankus}, while another one appears in the work \cite{DasII} by Das et al. To avoid any confusion with the existing literature, we briefly discuss these two definitions here. Let $p \in \C[z_1, z_2]$. Then $p$ is called \textit{toral in the sense of Agler et al.} if $\T^2$ is a determining set for the zero set $Z(p)$ of $p$, i.e., if $f$ is a holomorphic function on $Z(p)$ and $f|_{Z(p) \cap \T^2}=0$, then $f=0$ on $Z(p)$. On the other hand, if $Z(p)$ defines a distinguished variety with respect to $\D^2$, then $p$ is called \textit{toral in the sense of Das et al.}. Let us consider, the polynomial $p(z_1, z_2)=z_1-z_2$ whose zero set is $Z(p)=\{(z, z) : z \in \C \}$. Then it is not difficult to see that $p$ is toral in the sense of Das et al. Moreover, if $f$ is holomorphic on $Z(p)$, then the map $\widehat{f}: \C \to \C, \ \widehat{f}(z)= f(z, z)$ is an entire function. Let $f|_{Z(p) \cap \T^2}=0$. Then $\widehat{f}=0$ on $\T$ and, by identity theorem, $\widehat{f}=0$. Consequently, $f=0$ on $Z(p)$. Thus, $p$ is also toral in the sense of Agler and co-authors. Hence, these two classes of polynomials have a non-empty intersection. However, they are not same, i.e. these two notions of toral polynomials do not coincide. The polynomial $q(z_1, z_2)=1-z_1z_2$ is not toral in the sense of Das et al. as $Z(q) \cap \D^2 =\emptyset$. Since $Z(q)$ is disjoint from $\D^2 \cup (\C \setminus \DC)^2$, it follows from Theorem 3.5 in \cite{AglerMcCarthyStankus} that $p$ is toral in the sense of Agler et.al. In this article, we adopt the definition of toral polynomial in the sense of Das et al. as in \cite{DasII}.

\smallskip

 The following important class of polynomials appeared in the literature, e.g. see \cite{AglerKneseMcCarthy2, Knese9}.
\begin{defn}
	A polynomial $q \in \C[z_1, z_2]$ is said to be \textit{inner toral} if $Z(q) \subseteq \mathbb{D}^2 \cup \mathbb{T}^2 \cup \mathbb{E}^2$. 
\end{defn}
%-----------------------------------------------------------------------------------------
Needless to mention, every inner toral polynomial $q$ satisfies $Z(q) \cap \partial \D^2=Z(q) \cap \T^2$. Thus, an inner toral polynomial is close to being a toral polynomial. Indeed, the product of an inner toral polynomial with a toral polynomial gives a toral polynomial. However, we mention that not every toral polynomial is inner toral as the following example from \cite{DasII} clarifies.

\begin{eg}
	Consider $q(z_1, z_2)=z_1z_2-1$. Then $Z(q) \cap \partial \D^2 \subseteq \T^2$. So, the polynomial $q_0(z_1, z_2)=(z_1+z_2)q(z_1, z_2)$ is toral. Since $(2, 1\slash 2) \in Z(q_0) \cap (\mathbb{E} \times \D)$,  $q_0(z_1, z_2)$ is not inner toral. \qed
\end{eg} 

Below we define a natural analogue of inner toral polynomial for the symmetrized bidisc.

\begin{defn}
A polynomial $p \in \C[z_1, z_2]$ is said to be \textit{distinguished} if $Z(p)$ is a subset of $\mathbb{G}_2 \cup b\Gamma \cup \pi(\mathbb{E}^2)$.
\end{defn} 
One can construct distinguished and $\Gamma$-distinguished polynomials from inner toral polynomials and toral polynomial respectively via symmetrization. The subsequent example explains this.
%-----------------------------------------------------------------------------------------
\begin{eg} \label{3.3}
 Let $q(z_1, z_2)$ be a toral polynomial. Evidently, the symmetric polynomial $q_s(z_1, z_2)=q(z_1, z_2)q(z_2, z_1)$ is also toral. Since $q_s$ is a symmetric polynomial, there exists $p \in \C[z_1, z_2]$ such that $q_s=p\circ \pi$. It can be easily verified that $p$ is a $\Gamma$-distinguished polynomial. \qed 
 \end{eg}

\smallskip 
 
Conversely, if $p$ is a $\Gamma$-distinguished polynomial, then the polynomial $q=p \circ \pi$ is symmetric and toral as the points in the symmetrized bidisc are just symmetrization of the points of $\mathbb D^2$.

\begin{eg} \label{3.4}
	The polynomial $q(z_1, z_2)=z_1- z_2$ is a toral polynomial whose zero set has empty intersection with $\partial \overline{\mathbb D}^2 \setminus \T^2$. Evidently, $q_s(z_1, z_2)=(z_1- z_2)(z_2- z_1)$ is a symmetric toral polynomial with $Z(q_s) \cap \D^2 \neq \emptyset$. Thus, the polynomial $\widetilde q$ satisfying $q_s=\widetilde q \circ \pi$ is $\Gamma$-distinguished. Some routine calculations yield that $\widetilde q(z_1,z_2)=4z_2-z_1^2$ is $\Gamma$-distinguished. 
	\qed 
\end{eg}	

In this Section, we prove the existence of a square-free minimal annihilating polynomial for any algebraic pure $\Gamma$-isometry. One can ask if the zero set of this minimal polynomial is a distinguished variety in $\mathbb G_2$. We will see in the next Section that it is not true. We begin with the following notion. 
%-----------------------------------------------------------------------------------------	
	\begin{defn}\label{3.6}
		A $\Gamma$-isometry $(S,P)$ is called \textit{cyclic} if there exists $u \in \mathcal{H}$ such that the set
		\[
		\mathbb{C}[S,P] u:=\{q(S,P)u: q \in \mathbb{C}[z_1, z_2]\}
		\]
		is dense in $\mathcal{H}$. Such a vector $u \in \mathcal H$ is called a \textit{cyclic vector} for the cyclic $\Gamma$-isometry $(S,P)$.
		
	\end{defn}
%-----------------------------------------------------------------------------------------	
Given a contraction $T$ acting on a Hilbert space $\mathcal{H}$, it is well-known that $TD_T=D_{T^*}T$ (see Section 3 in \cite{NagyFoias6} for details).
As a consequence, we have the following elementary lemma.
\begin{lem}
For a Hilbert space contraction $T$, the following operator identity holds:
\begin{equation}\label{eqn_decomp}
\mathcal{D}_{T^*}=\overline{T\mathcal{D}_T}\oplus Ker \ T^*.
\end{equation}
\end{lem}

\begin{proof}
We have that $
T\mathcal{D}_T \subseteq \mathcal{D}_{T^*} $ and $T^*\mathcal{D}_{T^*} \subseteq \mathcal{D}_T$.
For any $x \in Ker \ T^*$, we have $x=x-TT^*x=D_{T^*}^2x$. Hence, $Ker \ T^* \subseteq D_{T^*}\mathcal{H}$. On the other hand, for every $x \in \mathcal{H}$ and $y \in Ker \ T^*$, we have
$
\langle TD_Tx, y\rangle = \langle D_Tx, T^*y\rangle=0.
$
So, $Ker \ T^*$ is orthogonal to $T\mathcal{D}_T$. Now, for any $g \in (T\mathcal{D}_T)^\perp \subseteq \mathcal D_{T^*}$ we have $T^*g \perp \mathcal{D}_T$. Again, $g \in \mathcal D_{T^*}$ and $T^*D_{T^*}=D_TT^*$ imply that $T^*g \in T^*\mathcal{D}_{T^*}\subseteq \mathcal D_T$.
Therefore, $T^*g=0$ and hence $g \in Ker \ T^*$ which proves that $(T\mathcal{D}_T)^\perp \subseteq \, Ker \ T^*$. So, we have (\ref{eqn_decomp}).
\end{proof}

%\smallskip

The following lemma shows how an annihilating polynomial gives an estimate of the dimension of the defect space of $P^*$ for a cyclic $\Gamma$-isometry $(S,P)$. In this connection, we say that a polynomial $q \in \mathbb{C}[z_1, z_2]$ has degree $(n, m)$ if it has degree $n$ in $z_1$ and $m$ in $z_2$.	 
%----------------------------------------------------------------------------------------	
	\begin{lem} \label{lem3.6}
		If $(S,P)$ is a cyclic $\Gamma$-isometry on a Hilbert space $\mathcal{H}$ satisfying $q(S,P)=0$, where $q$ has degree $(n, m)$, then dim $\mathcal{D}_{P^*} \leq n.$
	\end{lem}
	
	\begin{proof}
		Let $u\in \mathcal H$ be such that $\mathcal{H}=\overline{Span} \{p(S,P)u: p \in \mathbb{C}[z_1, z_2]\}$. Since $P$ is an isometry, we have $D_P=0.$ Thus, it follows from (\ref{eqn_decomp}) that $\mathcal{D}_{P^*}= \mbox{Ker} \, P^*$. If possible let $\operatorname{dim}$ Ker $P^{*}>n$. We can choose some non-zero $h \in \operatorname{Ker}\, P^{*}$ that is orthogonal to $S^{i}u$ for $i=0,1, \ldots, n-1$. Now, observe that
		$
		0=q(S,P)^{*}h=q(S, 0)^{*}h.
		$
		Let $p \in \mathbb{C}[z_1, z_2]$ and write
		$
		p(z_1, 0)=f(z_1)q(z_1, 0)+g(z_1),
		$
		where $g$ has degree less than $n$. Then,
		$
		p(S,P)^{*}h  =p\left(S, 0\right)^{*}h =f(S)^{*} q(S, 0)^{*}h+g\left(S\right)^{*}h  =g\left(S\right)^{*}h$.	Again, since $h$ is orthogonal to $S^iu$ for $i=0,1, \ldots, n-1$ and degree of $g<n,$  we have that 
		$
		\left\langle p(S,P) u,h \right\rangle=\left\langle g\left(S\right)u, h\right\rangle=0,
		$
		for any $p \in \mathbb{C}[z_1, z_2].$ Since $u$ is a cyclic vector, we have that $h=0$ and this leads to a contradiction. Hence, $\operatorname{dim} \mathcal D_{P^*}= \operatorname{dim} (\operatorname{Ker} \, P^{*}) \leq n$ and the proof is complete.		
	\end{proof}

For a Hilbert space $\HS$, the vectorial Hardy-Hilbert space $H^2(\HS)$ consists of all holomorphic functions from $\mathbb{D}$ to the Hilbert space $\HS$ with square summable coefficients, that is
\[
H^2(\HS)=\left\{\overset{\infty}{\underset{n=0}{\sum}}x_nz^n \ : \ z \in \D, \ x_n \in \HS \ \text{and} \  \overset{\infty}{\underset{n=0}{\sum}}\|x_n\|^2 < \infty \right\}.
\]
The \textit{Toeplitz operator} with symbol $\phi$, denoted by $T_\phi$, is defined for any bounded analytic function $\phi: \D \to \mathcal{B}(\HS)$ as the multiplication by $\phi$ on $H^2(\HS)$, i.e., $T_\phi(f)(z)=\phi(z)(f(z))$ for every $f \in H^2(\HS)$ and $z \in \D$. For $\phi(z)=zI$, we simply write $T_z$ instead of $T_\phi$. The rich operator theory of the symmetrized bidisc is based on one fundamental result from \cite{Pal8}. It states that for every $\Gamma$-contraction $(S,P)$, there is a unique operator $A \in \mathcal B(\mathcal D_P)$ with numerical radius $\omega(A)$ being not greater than $1$ such that $A$ satisfies the following operator equation in $X$:
\begin{equation} \label{eqn:funda-01}
S-S^*P=D_PXD_P.
\end{equation}
The unique operator $A$ is called the \textit{fundamental operator} of the $\Gamma$-contraction $(S,P)$. The following theorem gives an explicit model for a pure $\Gamma$-isometry in terms of the fundamental operator of its adjoint.
	\begin{thm}[\cite{PalShalit1}, Theorem 2.16]  \label{thm:modelpure}
		Let $(S,P)$ be a pair of commuting operators on a Hilbert space $\mathcal{H}$. If $(S,P)$ is a pure $\Gamma$-isometry, then there is a unitary operator $U:\mathcal{H} \to H^2(\mathcal{D}_{P^*})$ such that 
		\[
		S=U^*T_\phi U \quad \mbox{and} \quad P=U^*T_zU, \quad \mbox{where} \  \phi(z)=F_*^*+F_*z,
		\]
		$F_* \in B(\mathcal{D}_{P^*})$ being the fundamental operator of $(S^*, P^*)$.
	\end{thm}

The following theorem states that every pure $\Gamma$-contraction $(S,P)$ can be modeled as a compression of a pure $\Gamma$-isometry which is obtained in terms of the fundamental operator of $(S^*,P^*)$.
	
	\begin{thm}[\cite{Bhatt-Pal}, Theorem 3.1]\label{thm:2.12}
	Let $(S, P)$ be a pure $\Gamma$-contraction defined on a Hilbert space $\HS$. Then the operator pair $(T_{F_*^*+F_*z}, T_z)$ on $H^2(\mathcal{D}_{P^*})$  is a minimal isometric dilation of $(S, P)$. Here $F_*$ is the fundamental operator of $(S^*, P^*)$. Moreover, $S^*=T_{F_*^*+F_*z}^*|_\HS$ and $P^*=T_z^*|_\HS$.
\end{thm}
	
The next lemma is a consequence of Theorem \ref{thm:modelpure}.

	\begin{lem} 
		Let $(S,P)$ on a Hilbert space $\mathcal{H}$ be a pure $\Gamma$- isometry and let $F_*$ be the fundamental operator of $(S^*, P^*)$. If $\dim \mathcal{D}_{P^*} < \infty,$ then $(S,P)$ is algebraic. Moreover, if $\omega(F_*)<1,$ then $(S,P)$ is $\Gamma$-distinguished.	
	\end{lem}
	
	\begin{proof}
		By Theorem \ref{thm:modelpure}, a model for $(S,P)$ is the pair of multiplication operators $(T_{F_*^*+F_*z},T_z)$ on $H^2(\mathcal{D}_{P^*}).$ Since $\dim \mathcal{D}_{P^*} < \infty,$ we have that $F_*$ is a matrix. It follows from Theorem 3.5 in \cite{PalShalit1} that the polynomial $q(z_1, z_2)=det(F_*^*+F^*z_2-z_1I)$
is $\Gamma$-distinguished if $\omega(F_*)<1$. It can easily be verified that $q$ annihilates $(T_{F_*^*+F_*z},T_z)$.
	\end{proof}
	
	We now prove that every algebraic pure $\Gamma$-isometry has a minimal polynomial. We begin with the following lemma.
	
	\begin{lem} \label{3.9}
		Let $\left(S, P\right)$ be a cyclic pure $\Gamma$-isometry annihilated by an irreducible polynomial $q \in \mathbb{C}[z_1, z_2] .$ Then $q$ divides any polynomial $p$ that satisfies $p(S,P)=0$.	
	\end{lem}

	\begin{proof}
		
It follows from Theorem \ref{thm:modelpure}, that there is a unitary
operator operator $U:\mathcal{H} \to H^2(\mathcal{D}_{P^*})$ such that
                $
                S=U^*T_\phi U$ and $P=U^*T_zU$,
                where $\phi(z)=F_*^*+F_*z$
                with $F_* \in B(\mathcal{D}_{P^*})$ being the fundamental
operator of
$(S^*,P^*)$. By Lemma \ref{lem3.6}, the space $\mathcal{D}_{P^*}$ is
finite-dimensional. Thus $\phi$ is a matrix-valued linear polynomial. It
is easy to see that
                $
                q(T_\phi, T_z)=Uq(S, P)U^*=0$. Let the polynomial $q$ be given by $q(z_1,
z_2)=\overset{n}{\underset{i,j=0}{\sum}}q_{ij}z_1^iz_2^j$. Since $z_1I$
and $\phi(z_2)$ is a pair of commuting matrices acting on
$\mathcal{D}_{P^*},$ we have
                \begin{equation*}
                        \begin{split}
                                q(z_1, z_2)I=q(z_1I,
z_2I)-q(T_\phi,T_z)&=q(z_1I,
z_2I)-q(\phi(z_2)
,z_2I)\\
                                &=\overset{n}{\underset{i,j=0}{\sum}}q_{ij}z_1^iz_2^jI-\overset{n}{\underset{i,j=0}{\sum}}q_{ij}\phi(z_2)^iz_2^j\\
                                &=\overset{n}{\underset{j=0}{\sum}}\ \
\overset{n}{\underset{i=1}{\sum}}q_{ij}(z_1^i-\phi(z_2)^i)z_2^j\\
                                &=(z_1I-\phi(z_2))Q(z_1, z_2),\\
                        \end{split}
                \end{equation*}
                for some matricial polynomial $Q(z_1, z_2)$. The last
equality follows
from the fact that
                \[
                A^n-B^n=(A-B)(A^{n-1}+A^{n-2}B+ \dotsc + AB^{n-2}+B^{n-1})
                \]
                for any pair of commuting matrices $(A, B)$ and for any $n
\in
\mathbb{N}$. Now, $Q$ is not identically zero because, otherwise $q$
would also be equal to $0$ then, and consequently $Q$ would have a lower
degree in $z_1$ than $q$. So, the non-zero entries of the matrix
polynomial $Q$ cannot vanish identically on $Z(q)$. Since
$(z_1I-\phi(z_2)) Q(z_1, z_2) =q(z_1, z_2) I  =0 $ on $Z(q)$, we have
$z_1Q(z_1, z_2)=\phi(z_2) Q(z_1, z_2)$ on $Z(q)$.
                So, if $p \in \mathbb{C}[z_1, z_2]$ annihilates $(S,P)$
and hence
$(T_\phi,T_z)$, then
                \[
                p(\phi(z_2), z_2I)Q(z_1, z_2)=p(z_1, z_2) Q(z_1, z_2)=0
\quad \text { on
} Z(q).
                \]
Since $q$ is irreducible and $Q$ does not vanish identically on $Z(q)$, we
have that $q$ divides $p$.		 
	\end{proof}
	
The following theorem holds in general for a cyclic subnormal pair (e.g. see \cite{JohnMcCarthy3, Conway4, ConwayJohn5}) and thus holds in particular for a cyclic $\Gamma$-isometry. However, we state the theorem here for a cyclic $\Gamma$-isometry for our purpose. Let us mention here that for any compactly supported measure $\mu$ in $\mathbb{C}^{2}$, $P^{2}(\mu)$ denotes the closure of the polynomials in $L^{2}(\mu)$.
	%-------------------------------------------------------------------------------------
	\begin{thm} \label{3.10}
		Let $(S,P)$ be a cyclic $\Gamma$-isometry on the Hilbert space $\mathcal{H}$, with cyclic vector $u$. Then there is a positive Borel measure $\mu$ on some compact set $K$ in $\mathbb{C}^2$ and a unitary operator $U$ from $\mathcal{H}$ onto $P^{2}(\mu)$ that maps $u$ to the constant function $1$, and such that $U$ intertwines $(S,P)$ with the pair $\left(M_{z_1}, M_{z_2}\right)$ of multiplication by the coordinate functions.	
	\end{thm}
	%	DO MENTION THAT THIS IS A SPECIAL CASE OF THE (REFERRED) GENERAL THEOREM.\\
	
	Theorem \ref{3.10} makes it easy to show that the minimal polynomial of an algebraic pure $\Gamma$-isometry is square-free which we show below.
	
	\begin{lem} 
		Suppose $(S,P)$ is a pure $\Gamma$-isometry which is annihilated by $p={\displaystyle \prod_{i=1}^n p_{i}^{t_{i}}}$, where the irreducible factors of $p$ are $p_{1}, \dots , p_n$ with multiplicities $t_{1}, \dots t_n$ respectively.	Let $q={\displaystyle \prod_{i=1}^n p_{i}}$. Then $q(S,P)=0$.
	\end{lem}
	
	\begin{proof} Let $(S, P)$ be a pure $\Gamma$-isometry acting on a Hilbert space $\HS$. We need to show that $q(S,P)u=0$ for every $u \in \mathcal{H}$. For an arbitrary vector $u$ in $\mathcal{H}$, consider the space $	\mathcal{K}=\overline{\mathbb{C}[S,P] u}$.
		Suppose $(S',P')=(S,P)|_{\mathcal K}$. By Theorem \ref{3.10}, there is a unitary $U : \HS \to P^2(\mu)$ such that  $(S',P')=(U^*M_{z_1}U, U^*M_{z_2}U)$ for some positive Borel measure $\mu$ on some compact set. Thus
		\[
		0=p(S,P)=p(S',P')=p(M_{z_1}, M_{z_2}).
		\]
		In particular, $p(M_{z_1}, M_{z_2})1=p(z_1, z_2)=0$ on the support of $\mu$. Note that $q$ contains each irreducible factor of $p$. Therefore, $q$ vanishes identically on the support of $\mu$, and so we have 
		\begin{equation*}
			\begin{split}
				\|q(S,P)u\|^2&=\|q(U^*M_{z_1}U, U^*M_{z_2}U)u\|^2\\
				&=\|U^*q(M_{z_1}, M_{z_2})Uu\|^2\\
				&=\|q(M_{z_1}, M_{z_2})f\|_{P^2(\mu)}^2 \quad [f=Uu] \\
				&=\int|q(M_{z_1}, M_{z_2})f|^{2} d \mu\\
				&=\int|(q(M_{z_1}, M_{z_2}).1)(f)|^{2} d \mu\\
				& \leq \int|q(z_1, z_2)|^{2} d \mu . \int|f|^{2} d \mu \ \\
				&=0 \quad [\because q=0 \text{ on the support of } \mu].\\
			\end{split}
		\end{equation*}
		Hence, $q(S,P)u=0$. Since $u$ is arbitrary, we have $q(S,P)=0$. This finishes the proof.
	\end{proof}
	
	\begin{lem}\label{3.12}
		Suppose a pure $\Gamma$-isometry $(S,P)$ is annihilated by a polynomial $q$, where $q$ is a product of distinct irreducible factors and $(S,P)$ is not annihilated by any factor of $q$. Then $q$ divides any polynomial that annihilates $(S,P)$.
	\end{lem}
	\begin{proof} 
		Let $q_{0}$ be an irreducible factor and write $q=q_{0} q_{1}$. Since, $(S,P)$ is not annihilated by any of these factors, there exists $u_0 \in \mathcal{H}$ such that $u:=q_{1}(S,P) u_{0} \neq 0$. Let $\mathcal K$ be the cyclic subspace generated by $u$, i.e.,
		\[
		\mathcal{K}:=\overline{\mathbb{C}[S,P] u}=\overline{\mbox{span}}\{g(S,P) u: g \in \mathbb{C}[z_1,z_2]\}.
		\]
		Needless to mention, $\mathcal K$ is a joint invariant subspace for $(S,P)$. Let $(S',P')=\left (S|_\mathcal{K}, \ P|_\mathcal{K} \right)$, which is a pure cyclic $\Gamma$-isometry annihilated by $q_0$. By Lemma \ref{3.9}, $q_0$ divides every polynomial that annihilates $(S',P')$. If $g(S,P)=0$ for some $g \in \mathbb{C}[z_1, z_2]$, then $g(S',P')=0$ and Lemma \ref{3.9} implies that $q_{0}$ divides $g $. Since $q_{0}$ is arbitrary, every irreducible factor of $q$ divides $g$. Hence $q$ divides $g$ and the proof is complete.		
	\end{proof}
		
Combining all these results, we arrive at the following main theorem of this Section.
	
	\begin{thm} \label{thm:0312}
		Let $(S,P)$ be an algebraic pure $\Gamma$-isometry. Then there exists a square-free polynomial $q$ that annihilates $(S,P) $. Moreover, if $p$ is any polynomial that annihilates $(S,P)$, then $q$ divides $p$.	
	\end{thm}
	
	\section{The symmetrization map and algebraic $\Gamma$-isometries}\label{sec04}
	
	%\vspace{0.2cm}

\noindent 	The literature (e.g. see Theorem \ref{1.2}) tells us that every algebraic pair of commuting pure isometries is annihilated by a toral polynomial. A natural question arises, if every algebraic pure $\Gamma$-isometry is also annihilated by a $\Gamma$-distinguished polynomial. The next example shows that this is not true in general. The main reason behind this is that not every pure $\Gamma$-isometry arises as a symmetrization of a pure isometric pair. We justify this by an example after the following lemma.
%-----------------------------------------------------------------------------------------
\begin{lem}\label{q_A,0}	
Let $A \in \mathcal{B}(\mathcal H)$ for some Hilbert space $\HS$. If the pair $(T_{A+A^*z}, T_z)$ on $H^2(\mathcal H)$ is annihilated by a polynomial $q \in \C[z_1, z_2]$, then  $q(A, 0)=0$.
\end{lem}
%--------------------------------------------------------

\begin{proof}
For any $h_0 \in \mathcal{H}$, the map $f_0: \D \to \mathcal{H}$ given by $f_0(z)=h_0$ is in $H^2(\mathcal{H})$. By definition, $T_z(f_0)(0)=0$ and $T_{A+A^*z}(f_0)(0)=Af_0(0)=Ah_0$. Thus, $q(A, 0)h_0=q(T_{A+A^*z}, T_z)(f_0(0))=0$ and so $q(A, 0)=0$.
\end{proof} 
%-----------------------------------------------------------
	\begin{eg}\label{3.14}
		Consider the matrix
		\[
		A=\begin{pmatrix}
			0 & 2 & 0\\
			0 & 0 & 0\\
			0 & 0 & 1\\
		\end{pmatrix}.
		\]
		Then the numerical radius $\omega(A)$ is equal to $1$. By Theorem \ref{thm:modelpure}, the Toeplitz pair $(T_{A+A^*z}, T_z)$ on $H^2(\mathbb{C}^3)$ is a pure $\Gamma$-isometry. It can be easily verified that the polynomial $p(z_1, z_2)=det(A+A^*z_2-z_1I)$ annihilates $(T_{A+A^*z}, T_z)$. If possible let $(T_{A+A^*z}, T_z)$ be annihilated by a $\Gamma$-distinguished polynomial $q(z_1, z_2)$. Then $q(A, 0)=0$ by Lemma \ref{q_A,0}. It follows from spectral mapping theorem that
		\[
		\{ 0 \} =\sigma(q(A,0))=\{ q(\lambda , 0)\,:\, \lambda \in \sigma(A) \}.
		\]
		Since $\sigma(A)=\{0,1\},$ we have that $q(1,0)=0$. Note that the point $(1,0) \in \partial \Gamma = \Gamma \setminus \mathbb G_2$ by being the symmetrization of the points $0,1$. The fact that $q$ is $\Gamma$-distinguished implies that 
		$
		(1,0) \in Z(q) \cap \partial \Gamma \subseteq b\Gamma $. This is a contradiction as $(1,0) \notin b\Gamma$, because, every point in $b\Gamma$ is of the form $(z_1+z_2,z_1z_z)$ with unimodular $z_1,z_2$. Hence, $(T_{A+A^*z}, T_z)$ is an algebraic pure $\Gamma$-isometry but no $\Gamma$-distinguished polynomial can annihilate it. \qed
	\end{eg}
	
As we have mentioned above that the reason behind an algebraic pure
$\Gamma$-isometry being not $\Gamma$-distinguished in general is that not
every $\Gamma$-isometry arises as the symmetrization of commuting
isometries. The pair $(T_{A+A^*z}, T_z)$, in Example \ref{3.14}, is a
pure $\Gamma$-isometry acting on $H^2(\C^3)$ and it is annihilated by
$p(z_1, z_2)=det(A+A^*z_2-z_1I)$. We prove that $(T_{A+A^*z}, T_z)$
cannot be the symmetrization of a pair of commuting pure isometries. Let
if possible, $(T_{A+A^*z}, T_z)=\pi(V_1, V_2)$ for a pair of commuting
pure isometries $(V_1, V_2)$ on $H^2(\mathbb{C}^3)$. Then the polynomial
$p\circ \pi$ annihilates $(V_1, V_2)$. We have by Theorem \ref{1.2} that
there exists $q \in \mathbb{C}[z_1, z_2]$ such that $q$ annihilates
$(V_1, V_2)$ and $Z(q) \subseteq \mathbb{D}^2 \cup \mathbb{T}^2 \cup
\mathbb{E}^2$. Thus, $Z(q) \cap \partial \overline{\mathbb{D}}^2 \subseteq
\mathbb{T}^2 $. The polynomial $q_s(z_1, z_2)=q(z_1, z_2)q(z_2, z_1)$
also annihilates $(V_1, V_2)$ with $Z(q_s) \subseteq \mathbb{D}^2 \cup
\mathbb{T}^2 \cup \mathbb{E}^2$. Since $q_s$ is symmetric, one can choose
$\Tilde{q} \in \mathbb{C}[z_1, z_2]$ such that $q_s=\Tilde{q}\circ \pi$.
By Example \ref{3.3}, $\Tilde{q}$ is distinguished and so, $Z(\Tilde{q})
\cap \partial \mathbb{G}_2 \subseteq b\Gamma$. Moreover,
$\Tilde{q}(T_{A+A^*z}, T_z)=\Tilde{q}\circ \pi(V_1, V_2)=q_s(V_1,
V_2)=0$. By Lemma \ref{q_A,0}, $\Tilde{q}(A,0)=0$. Let $f(z_1,
z_2)=z_1\Tilde{q}(z_1, z_2)$. Evidently, $(0, 0) \in Z(f) \cap
\mathbb{G}_2$. Since $Z(f)=(\{0\} \times \C )\cup Z(\Tilde{q})$, we have
for any $(s, p) \in Z(f) \cap \partial \Gamma$ that either $(s, p)
\in Z(\tilde{q}) \cap \partial \Gamma \subseteq b\Gamma$ or $s=0$
and $(0, p) \in \partial \Gamma$. In the latter case, there exists
$(z_1, z_2) \in \partial \overline{\D}^2=(\DC \times \T) \cup (\T \times \DC)$ such
that $z_1+z_2=0$ and $z_1z_2=p$. Thus, $|z_1|=|z_2|=1$ and so, $(0,
p)=\pi(z_1, -z_1)\in \pi(\T^2)=b\Gamma$. Consequently, $f$ is a
$\Gamma$-distinguished polynomial and that $f(A, 0)=0$. This gives a 
contradiction to the fact that $(A, 0)$ cannot be annihilated by any
$\Gamma$-distinguished polynomial as shown in Example \ref{3.14}.
Therefore, $(T_{A+A^*z}, T_z)$ cannot arise as the symmetrization of a
commuting pair of pure isometries.

        \medskip

        The symmetrization of a pair of commuting pure isometries is always a
pure $\Gamma$-isometry but the converse does not hold as discussed above.
Also, we present below a simple example showing that not every pure
$\Gamma$-isometry arises as the symmetrization of commuting isometries.
	
	\begin{eg}
		Consider the pair $(0,T_z)$ on $H^2(\mathbb{D})$ which is annihilated by the polynomial $q(z_1, z_2)=z_1$. Agler and Young proved in \cite{AglerII16} that a commuting pair of operators $(T,V)$ is a pure $\Gamma$-isometry if and only if $\|T\|\leq 2$, $T=T^*V$ and $V$ is a pure isometry, i.e., unitarily equivalent to a shift operator on a Hardy space. It follows from here that $(0,T_z)$ is a pure $\Gamma$-isometry. If this pair is symmetrization of a pair of commuting isometries, say $(V_1, V_2)$ on $H^2(\mathbb{D}),$ then $V_1+V_2=0$ and $V_1V_2=T_z$ and consequently we have $-V_2^2=T_z$. This is a contradiction, because the shift operator $T_z$ cannot be written as a square of an operator. \qed
			\end{eg}
	
	The above examples lead to the question, if we can characterize all $\Gamma$-isometries which are symmetrization of commuting isometries. We borrow techniques as in \cite{PalSourav11} by the first named author of this article and obtain a necessary and sufficient condition for the same.	
	
	\begin{thm} 
		Let $(S,P)$ be a $\Gamma$-isometry acting on a Hilbert space $\mathcal{H}$. Then $S=V_1+V_2$ and $P=V_1V_2$ for a pair of commuting isometries $V_1, V_2$ on $\mathcal{H}$ if and only if $S^2-4P$ has a square root $\Delta$ such that $\Delta$ commutes with $S,P$ and $\frac{1}{2}(S \pm \Delta)$ are isometries.
	\end{thm}

	\begin{proof} 
		Let $(V_1, V_2)$ be two commuting isometries such that $S=V_1+V_2$ and $P=V_1V_2$. Then $S^2-4P=(V_1-V_2)^2$ and hence, $S^2-4P$ has a square root, say, $\Delta=V_1-V_2$ which commutes with $S$ and $P$. Clearly, 
		\[
		\left(\frac{1}{2}(S+ \Delta), \frac{1}{2}(S- \Delta)\right)=(V_1, V_2),
		\]
		which is a commuting pair of isometries. Conversely, suppose $S^2-4P$ has a square root $\Delta$ that commutes with $S,P$ and $\frac{1}{2}(S \pm \Delta)$ are isometries. Set 
		$
			V_1=\frac{1}{2}(S + \Delta)$ and $V_2=\frac{1}{2}(S - \Delta)$. Then $(V_1, V_2)$ is a commuting pair of isometries such that $S=V_1+V_2$ and $P=V_1V_2$.
	\end{proof}

	The following corollary is an easy consequence of the above theorem.
	
	\begin{cor} \label{cor:044}
		
		Let $(S,P)$ be a pure $\Gamma$-isometry acting on a Hilbert space $\mathcal{H}$. Then $S=V_1+V_2$ and $P=V_1V_2$ for a pair of commuting pure isometries $V_1, V_2$ on $\mathcal{H}$ if and only if $S^2-4P$ has a square root $\Delta$ such which commutes with $S$ and $P$ and $\frac{1}{2}(S \pm \Delta)$ are pure isometries.
	\end{cor} 
%-----------------------------------------------------------------------------------------	
	The conditions that $\frac{1}{2}(S \pm \Delta)$ are pure in Corollary \ref{cor:044} cannot be ignored and the following example explains this.
%-----------------------------------------------------------------------------------------	
	\begin{eg}
		Let us consider the following matrices: 
		\[
			A=\begin{pmatrix}
				0 & 2 & 0 \\
				0 & 0 & 0\\
				0 & 0 & 1\\
			\end{pmatrix} \quad \text{and} \quad E=\begin{pmatrix}
			0 & 0 & 0 \\
			0 & 0 & 0\\
			0 & 0 & 1\\
		\end{pmatrix}.
		\]
		As mentioned in Example \ref{3.14} and discussion thereafter, $\omega(A)=1$ and the pure $\Gamma$-isometry $(S,P)=(T_{A+A^*z},T_z)$ on $H^2(\mathbb{C}^3)$
		cannot arise as the symmetrization of a commuting pair of pure isometries. 
		Here, we show that $S^2-4P$ has a square root $\Delta$ which commutes with $S, P$ and $\frac{1}{2}(S \pm \Delta)$ are commuting isometries. It is evident that $A^2=A^{*2}=E$ and $A^*A+AA^*+2E=4I$. Define $\Delta:=T_{E-Ez}$ on $H^2(\mathbb{C}^3)$ whose replica in $l^2(\C^3)$ is given by the following block matrix:
		\begin{equation*}
			\Delta=  \begin{bmatrix} 
				E & 0 & 0 & 0 & \dotsc \\
				-E & E & 0 & 0 & \dotsc \\
				0 & -E & E & 0 & \dotsc\\
				0 & 0 & -E & E & \dotsc\\
				\dotsc & \dotsc & \dotsc & \dotsc & \dotsc \\
			\end{bmatrix}.
		\end{equation*}
		We now show that $S^2-4P=\Delta^2$. 
		\begin{equation*}
			\begin{split}
				S^2-4P&=  \begin{bmatrix} 
					A & 0 & 0 & 0 & \dotsc \\
					A^* & A & 0 & 0 & \dotsc \\
					0 & A^* & A & 0 & \dotsc\\
					0 & 0 & A^* & A & \dotsc\\
					\dotsc & \dotsc & \dotsc & \dotsc & \dotsc \\
				\end{bmatrix}  \begin{bmatrix} 
					A & 0 & 0 & 0 & \dotsc \\
					A^* & A & 0 & 0 & \dotsc \\
					0 & A^* & A & 0 & \dotsc\\
					0 & 0 & A^* & A & \dotsc\\
					\dotsc & \dotsc & \dotsc & \dotsc & \dotsc \\
				\end{bmatrix}-4  \begin{bmatrix} 
					0 & 0 & 0 & 0 & \dotsc \\
					I & 0 & 0 & 0 & \dotsc \\
					0 & I & 0 & 0 & \dotsc\\
					0 & 0 & I & 0 & \dotsc\\
					\dotsc & \dotsc & \dotsc & \dotsc & \dotsc \\
				\end{bmatrix}\\
				&=  \begin{bmatrix} 
					A^2 & 0 & 0 & 0 & \dotsc \\
					A^*A+AA^*-4I & A^2 & 0 & 0 & \dotsc \\
					A^{*2} & A^*A+AA^*-4I & A^2 & 0 & \dotsc\\
					0 & A^{*2} & A^*A+AA^*-4I & A^2 & \dotsc\\
					\dotsc & \dotsc & \dotsc & \dotsc & \dotsc \\
				\end{bmatrix}\\
				&=\begin{bmatrix} 
					E & 0 & 0 & 0 & \dotsc \\
					-2E & E & 0 & 0 & \dotsc \\
					E & -2E & E & 0 & \dotsc\\
					0 & E & -2E & E & \dotsc\\
					\dotsc & \dotsc & \dotsc & \dotsc & \dotsc \\
				\end{bmatrix}=\Delta^2.\\
			\end{split}
		\end{equation*}
Thus, $S^2-4P$ has a square root $\Delta$. Again, a few steps of calculation give $S\Delta=\Delta S$. Similarly, one can show that $\Delta P=P \Delta$. It is now easy to see that 
\begin{small}	
		\begin{equation*}
			S+\Delta=  \begin{bmatrix} 
				A+E & 0 & 0 & 0 & \dotsc \\
				A^*-E & A+E & 0 & 0 & \dotsc \\
				0 & A^*-E & A+E & 0 & \dotsc\\
				0 & 0 & A^*-E & A+E & \dotsc\\
				\dotsc & \dotsc & \dotsc & \dotsc & \dotsc \\
			\end{bmatrix} \ , \  S-\Delta= \begin{bmatrix} 
				A-E & 0 & 0 & 0 & \dotsc \\
				A^*+E & A-E & 0 & 0 & \dotsc \\
				0 & A^*+E & A-E & 0 & \dotsc\\
				0 & 0 & A^*+E & A-E & \dotsc\\
				\dotsc & \dotsc & \dotsc & \dotsc & \dotsc \\
			\end{bmatrix}.
		\end{equation*}
	\end{small}
	It follows from here that $(S+\Delta)^*(S+\Delta)= 4I$. Similarly, we have $(S-\Delta)^*(S-\Delta)=4I$. \qed	
	\end{eg}
	
		\section{The distinguished boundary of a distinguished variety in $\D^2$ and $\G$}\label{sec06}
		
	%\vspace{0.2cm}
	
	\noindent Recall that the distinguished boundary $bX$ of a compact set $X \subset \C^n$ is the smallest closed subset of $X$ on which every member in $Rat(X)$ attains its maximum modulus. Also, $bX$ is the Shilov boundary of the algebra $Rat(X)$. It follows from the maximum principle that $bX$ is always contained in the topological boundary $\partial X$. However, $bX$ can be thinner than $\partial X$ in higher dimensions. For example, the closed disc $\overline{\D}$ has distinguished boundary $\T$ which is also its topological boundary but for the closed bidisc $\overline{\D}^2$ the topological boundary is $(\overline{\D} \times \T) \cup (\T \times \overline{\D})$ which is much bigger than the distinguished boundary $\T^2$. Clearly, if there is a function $f \in Rat(X)$ and a point $x \in X$ such that $f(x)=1$ and $|f(y)|<1$ for all $y \in X \setminus \{x\}$, then $x \in bX$. Such a point $x$ is said to be a \textit{peak point} of $X$ and the function $f$ is called a \textit{peaking function} for $x$.
	
	\smallskip

	The distinguished boundary plays a significant role in both complex-function theory and operator theory associated with a domain. A seminal work due to Arveson (e.g. see Corollary to Theorem 1.2.2 in \cite{ArvesonI} or Theorem \ref{thm_Arveson} of this paper) states that a commuting operator tuple $\underline{T}=(T_1, \dots , T_n)$ has $X$ as a complete spectral set if and only if $\underline{T}$ possesses a normal dilation $\underline{N}=(N_1, \dots , N_n)$ such that the Taylor joint spectrum $\sigma_T(\underline{N}) \subseteq bX$. We intend to study Arveson's theorem when a pair of toral contractions or a $\Gamma$-distinguished $\Gamma$-contraction admits a toral unitary dilation or a $\Gamma$-distinguished $\Gamma$-unitary dilation respectively. Note that such a study, even if succeeds, does not guarantee the success or failure of rational dilation on a distinguished variety in $\overline{\D}^2$ or $\overline{\mathbb G}_2$. However, before getting into complete spectral set versus normal $bX$-dilation in our setting it is necessary to determine the distinguished boundary of a distinguished variety in $\D^2$ and $\G$.

\smallskip 

 There are several techniques in the literature (e.g. see \cite{Oka-Weil}) to determine distinguished boundaries of bounded domains in $\C^n$. However, these techniques vary from domain to domain and till date there is no fixed algorithm to find the distinguished boundary of any compact set in $\C^n$. Unfortunately, none of these techniques seem to work for a distinguished variety in $\D^2$ or $\G$. A probable underlying reason is that the compact sets of type $Z(p) \cap \overline{\D}^2$ or $Z(p) \cap \Gamma$ are too thin. Here we shall apply the theory of analytic variety in a domain \cite{Rossi, Stanislaw} for our purpose. Indeed, by an application of this theory we shall show in this section that $b(Z(p)\cap \overline{\D}^2)=Z(p) \cap \T^2$ and $b(Z(p) \cap \Gamma) =Z(p) \cap b\Gamma$, when $p$ is a toral or $\Gamma$-distinguished polynomial respectively. We begin with a brief theory of analytic variety in a domain.
  
 \subsection{Analytic variety and maximum principle} Let $\Omega \subseteq \C^n$ be a domain. A subset $V$ is said to be an \textit{analytic variety} or an \textit{analytic subset} or a \textit{subvariety} of $\Omega$ if for every point $z \in \Omega$, there is a neighborhood $U_z$ of $z$ and functions $f_1, \dotsc, f_k$ holomorphic in $U_z$ such that 
\[
U_z \cap V=\left\{w \in U_z \ : \ f_1(w)=0, \dotsc, f_k(w)=0 \right\}.
\]
Let $M$ be a complex manifold. By a \textit{globally analytic variety} or a \textit{globally analytic subset} of the manifold $M$, we mean any set of the form 
\[
V=\left\{z \in M \ : \ f_1(z)=0, \dotsc, f_k(z)=0 \right\},
\]
where $f_1, \dotsc, f_k$ are holomorphic on $M$. A subset $Z$ of a manifold $M$ is called an \textit{analytic subset} of $M$ if every point of $M$ has an open neighborhood $U$ such that the set $Z \cap U$ is a global analytic subset of $U$. Since every domain $\Omega$ in $\C^n$ is a complex manifold of dimension $n$, one can define an analytic subset of $\Omega$. It is easy to see that, in this case, the concept of an analytic subset of $\Omega$ coincides with that of a subvariety of $\Omega$. An interested reader is referred to Chapter II, Section 3 in \cite{Stanislaw} and  Chapter II, Section E in \cite{Rossi} for a detailed study. We mention a few important definitions and results which will be used in sequel.

\begin{lem}[\cite{Stanislaw}, Chapter IV, Section 1]\label{component} 
	Connected components of an analytic subset of a complex manifold $M$ are analytic subsets of $M$. 
\end{lem}

An immediate consequence of Lemma \ref{component} is that if $V$ is a subvariety of a domain $\Omega$, then each connected component of $V$ is again a subvariety of $\Omega $. 

\begin{defn}
	Let $f$ be a real-valued function defined on a domain $\Omega$ in $\C^n$. The function $f$ is said to be \textit{plurisubharmonic} if $f$ is upper semi-continuous and has the following property: if $\phi:\D \to \Omega$ is holomorphic, then $f \circ \phi$ is subharmonic.
\end{defn}

It is well-known that if $f$ is a holomorphic function on a domain $\Omega \subseteq \C^n$, then $|f|$ is a plurisubharmonic function on $\Omega$. We now mention that plurisubharmonic functions obey the following maximum principle. 

%---------------------------------------------------------------------------------------
\begin{thm}[\cite{Rossi}, Chapter IX, Section C, Proposition 3]\label{thm_max} Let $f$ be a continuous plurisubharmonic function defined in a domain $\Omega$, and let $V$ be a closed connected subvariety of $\Omega$. If $f|_V$ attains its maximum at some point of $V$, then $f|_V$ is constant.  
\end{thm}
%-------------------------------------------------------------------------------------
The following theorem, a consequence of Cartan's prominent work on analytic sheaves, was originally proved in \cite{Cartan}. Comprehensive proofs are provided in the classical texts \cite{Rossi} and \cite{Hormander} (Theorem 7.4.8). By a holomorphic function on an analytic variety $V$, we mean a function that locally agrees with the restriction of a holomorphic function on an open set containing $V$. 
%---------------------------------------------------------------------------------------
\begin{thm}[H. Cartan]\label{Cartan} If $V$ is an analytic variety in a domain of holomorphy $\Omega$ and if $f$ is a holomorphic function on $V$, there is a holomorphic function $F$ on $\Omega$ such that $F=f$ on $V$. 
\end{thm}

Begin armed with the theory of analytic variety in a domain, we now proceed to finding the distinguished boundary of a distinguished variety in the bidisc and the symmetrized bidisc. We begin with a simple lemma.

\begin{lem}\label{lem801} 
		Let $K$ be a polynomially convex set in $\C^n$ and let $p \in \C[z_1, \dotsc, z_n]$ be such that $Z(p) \cap K \ne \emptyset$. Then $Z(p) \cap K$ is a polynomially convex set in $\C^n$. 
	\end{lem}
	
		\begin{proof} Let $X=Z(p) \cap K$. To prove $X$ is polynomially convex, it suffices to show that for any $x \notin X$, there exists a polynomial $f \in \C[z_1, \dotsc, z_n]$ such that $\|f\|_{\infty, X}<|f(x)|$. Take any $x \notin X$. We discuss the two cases here depending on if $x$ lies in $K$ or not. Let $x \notin K$. Then there exists $f \in \C[z_1, \dotsc, z_n]$ such that $\|f\|_{\infty, K} <|f(x)|$ since $K$ is polynomially convex. Thus $\|f\|_{\infty, X} \leq \|f\|_{\infty, K} <|f(x)|$. For $x \in K$, we have $x \notin Z(p)$ and so, $|p(x)|>0=\|p\|_{\infty, X}$. The proof is complete. 
	\end{proof}
	%----------------------------------------------------------------------------------
	\begin{lem}\label{toral_peak}
		Let $q$ be a toral polynomial. Then	every point of $Z(q) \cap \mathbb{T}^2$ is a peak point of $Z(q) \cap \overline{\mathbb{D}}^2$.
	\end{lem}
	
	\begin{proof}
		Let $(e^{ix}, e^{iy}) \in Z(q) \cap \mathbb{T}^2$ for $x, y \in \mathbb{R}$. Then the map $g$ given by 
		\[
		g: \overline{\mathbb{D}}^2 \to \mathbb{C}, \quad g(z_1, z_2)=\left(\frac{e^{ix}}{2e^{ix}-z_1}\right)\left(\frac{e^{iy}}{2e^{iy}-z_2}\right)
		\]
		is holomorphic on $\overline{{\mathbb{D}}}^2$ and $g(e^{ix}, e^{iy})=1$. For every $z \in \overline{\mathbb{D}}$, we have that $|2e^{ix}-z| \geq 2-|z| \geq 1$ and so, 
		\[
		|g(z_1, z_2)|=\left|\frac{e^{ix}}{2e^{ix}-z_1}\right|\cdot\left|\frac{e^{iy}}{2e^{iy}-z_2}\right| \leq 1 \quad \text{for} \ z_1, z_2 \in \overline{\mathbb{D}}.
		\]	
		Assume that $|g(z_1, z_2)|=1$ for some $(z_1, z_2) \in \overline{\mathbb{D}}^2$. Then $|2e^{ix}-z_1|=1=|2e^{iy}-z_2|$ which is possible if and only if $(z_1, z_2)=(e^{ix}, e^{iy})$. Hence, $g$ is a peaking function for $(e^{ix}, e^{iy})$. The restriction of $g$ to $Z(q) \cap \overline{\mathbb{D}}^2$ is a map in $Rat(Z(q) \cap \overline{\mathbb{D}}^2)$ that peaks at $(e^{ix}, e^{iy})$ and the proof is complete.		
	\end{proof}
	
The following theorem determines the distinguished boundary of a distinguished variety in $\D^2$ or even more. This is a main result of this Section.
	
		\begin{thm}\label{thm803}
		Let $q \in \C[z_1, z_2]$ be a toral polynomial. Then the following are equivalent: 
		\begin{enumerate}
			\item $(z_1, z_2) \in Z(q) \cap \T^2$;
			\item $(z_1, z_2)$ is a peak point of $Z(q) \cap \DC^2$;
			\item $(z_1, z_2) \in b(Z(q) \cap \DC^2)$.
		\end{enumerate}
		Moreover,	$b(Z(q)\cap \DC^2)=Z(q) \cap \T^2$. 
	\end{thm}
	
	\begin{proof}
	
	We prove $(1) \Rightarrow (2) \Rightarrow (3) \Rightarrow (1)$. We shall use the following notations: $X=Z(q) \cap \DC^2$ and $V=Z(q)\cap \D^2$.
	
	\smallskip
		
		\noindent $(1) \Rightarrow (2)$. Let $(z_1, z_2) \in Z(q) \cap \T^2$. It follows from Lemma \ref{toral_peak} that there exists $f \in Rat(X)$ that peaks at $(z_1, z_2)$. 
		
		\vspace{0.2cm}
		
		\noindent $(2) \Rightarrow (3)$. It follows from the definition of peak point and the distinguished boundary that every peak point of $Z(q) \cap \DC^2$ must belong to $b(Z(q)\cap \DC^2)$.
		
		\vspace{0.2cm}
		
		\noindent $(3) \Rightarrow (1)$.  First we prove that $Z(q) \cap \T^2$ is a closed boundary of $X$, i.e., every $g \in Rat(X)$ attains its maximum modulus over $Z(q) \cap \DC^2$ at some point in $Z(q) \cap \T^2$. Take any $g \in Rat(X)$ and some $x \in X$ such that 
		$
		\|g\|_{\infty,\, X}=|g(x)|.
		$ 
		We show that there is $y \in Z(q) \cap \T^2$ such that $\|g\|_{\infty, X}=|g(y)|$. If $x \in Z(q) \cap \T^2$, then we are done. We assume that $x \notin Z(q) \cap \T^2$. Note that  
		\[
		X=Z(q) \cap (\D^2 \cup \partial \D^2)=V \cup (Z(q) \cap \partial \D^2)=V \cup (Z(q) \cap \T^2),
		\] 
		where the last equality follows from the fact that $q$ is toral. Thus $x \in V$. We can re-write 
		$
		V=\underset{i \in I}{\bigsqcup} \ V_i \,,
		$ 
		where each $V_i$ is a connected component of $V$ such that $V_i \cap V_j =\emptyset$ for all $i \ne j$. Indeed, each $V_i$ is a connected component of $V$ that exits $\D^2$ through its distinguished boundary, the $2$-torus $\mathbb{T}^2$. Therefore, the closure of each $V_i$ in $\mathbb{C}^2$, denoted by $\overline{V_i}$, must intersect $\T^2$. Since $V$ is an analytic subset or subvariety of $\D^2$, it follows from Lemma \ref{component} that each $V_i$ is a subvariety of $\D^2$. Also, it follows that each $V_i$ is a closed connected subvariety of $\D^2$. Since $x \in V$, we must have $x \in V_i$ for some $i \in I$. The function $g|_{V_i}$ is holomorphic on $V_i$ and hence, it follows from Theorem \ref{Cartan} that there is a holomorphic map $h$ on $\D^2$ such that 
		$
		h|_{V_i}=g|_{V_i}.
		$
		The function 
		$
		f: \D^2 \to \mathbb{R}$ defined by $f(z)=|h(z)|
		$
		is a continuous plurisubharmonic function on $\D^2$ such that $f|_{V_i}$ (which is same as $|g||_{V_i}$) attains its maximum at $x \in V_i$. It follows from Theorem \ref{thm_max} that $f|_{V_i}$ is constant and hence, 
		$
		|g(z)|=|g(x)| $ for all $z \in V_i$.	Using continuity argument, we have that $|g(z)|=|g(x)|$ for all $z \in \overline{V_i}$. Since $\overline{V_i} \cap \T^2 \ne \emptyset$, we have that $|g(x)|=|g(y)|$ for some $y \in \overline{V_i} \cap \T^2$. Hence, we have 
		$
		\|g\|_{\infty, X}=|g(x)|=|g(y)|
		$
		for some $y \in Z(q) \cap\T^2$. Since $g \in Rat(X)$ is arbitrary, every function in $Rat(X)$ attains its maximum modulus at some point in $Z(q) \cap \T^2$. Therefore, $Z(q) \cap \T^2$ is a closed boundary for $X$. Since $b(Z(q) \cap \DC^2)$ is the smallest closed boundary of $X$, we have $b(Z(q) \cap \DC^2) \subseteq Z(q) \cap \T^2$ and the proof is complete. 
  \end{proof}
	
The above theorem may fail if the concerned polynomial is not a toral polynomial. Below we show it by an example.
\begin{eg}\label{eg_toral}
The polynomial $p(z_1, z_2)=z_2$ is not toral. It is easy to see that
$
Z(p) \cap \DC^2=\DC \times \{0\}$ and $Z(p) \cap \T^2=\emptyset$. Take any $(e^{i\theta}, 0) \in  \T \times \{0\}$. Following the proof of Lemma \ref{toral_peak}, the rational function given by
\[
f: Z(p) \cap \DC^2 \to \C, \quad f(z_1, z_2)=\frac{e^{i\theta}}{2e^{i\theta}-z_1}
\]
peaks at $(e^{i\theta}, 0)$. Thus $(e^{i\theta}, 0) \in b(Z(p) \cap \DC^2)$ and so, $Z(p) \cap \T^2 \ne b(Z(p) \cap \DC^2)$.
\qed	
\end{eg}
Evidently, the distinguished boundary of the symmetrized bidisc $b\Gamma$ is the symmetrization of the $2$-torus $\T^2$ (see \cite{AglerII16} for details). The points in $b\Gamma$ was characterized in \cite{AglerII16} in several different ways. Also, by Proposition 8.1 in \cite{AglerLykova}, if $x \in b\Gamma$ then there is a function $f$ that is holomorphic on $\G$ and continuous on $\Gamma$ such that $f$ peaks at $x$. For our purpose, we need a peaking function $f$ in $Rat(\Gamma)$ and the next theorem provides such a function. We give a proof to this theorem here, because we could not locate it in the literature. However, our proof is based on the techniques that are used in the proof of Theorem 8.4 in \cite{AglerLykova}.

	\begin{thm}\label{thm::peak_Gamma}
		Let $(s_0, p_0) \in b\Gamma$. Then there exists $f \in Rat(\Gamma)$ such that $f$ peaks at $(s_0, p_0)$.	
	\end{thm}

	\begin{proof}
		Consider $(s_0, p_0)=(z_1+z_2, z_1z_2)$ for some $z_1, z_2 \in \T$. If $z_1=z_2$, then $(s_0, p_0)=(2z_1, z_1^2)$ and the function 
		$
		f(s, p)={1}\slash {2}(1+s\slash s_0)
		$
		peaks at $(s_0, p_0)$. To see this, note that $f(s_0, p_0)=1$ and $|f(s, p)| \leq \frac{1}{2}(1+|s|\slash 2) \leq 1$ for every $(s,p) \in \Gamma$. If $|f(s, p)|=1$ for $(s, p)=(\alpha_1+\alpha_2, \alpha_1\alpha_2) \in \Gamma$, then $|1+s\slash s_0|=2$. Since $s\slash s_0 \in \DC$, we must have $s_0=s$. Then 
		$
		2z_1=\alpha_1+\alpha_2 $ and so, $|\alpha_1|+|\alpha_2|=2$ which gives $|\alpha_1|=|\alpha_2|=1$, as $\alpha_1, \alpha_2 \in \DC$. Again, $\alpha_1+\alpha_2=2z_1$ implies that $|1+\alpha_1\overline{\alpha}_2|=2$, which is possible only when $\alpha_1\overline{\alpha}_2=1$. Thus, $\alpha_1=\alpha_2=z_1$ and hence, $(s, p)=(s_0, p_0)$. Let $z_1 \ne z_2$. Choose an automorphism $v$ of $\D$ such that $v(z_1)=1$ and $v(z_2)=-1$. The map
		\[
		T_v : \G \to \G, \quad T_v(w_1+w_2, w_1w_2)=\left(v(w_1)+v(w_2), v(w_1)v(w_2) \right)
		\] 
		defines an automorphism of $\G$ which extends continuously from $\Gamma$ onto $\Gamma$ in a bijective manner such that $T_v(b\Gamma)=b\Gamma$. Note that $T_v(s_0, p_0)=(0, -1)$. The map  
		\[
		g: \Gamma \to \C, \quad g(s, p)=\frac{1}{2}\left(1+\frac{s^2-4p}{4}\right).
		\]
		is a polynomial with $g(0, -1)=1$. It follows from Theorem 1.1 in \cite{AglerII16} that every $(s,p) \in \Gamma$ satisfies the following inequality:
		\begin{equation}\label{eqn:3.2}
			\frac{|s-\overline{s}p|}{2} +\frac{|s^2-4p|}{4}+\frac{|s|^2}{4} \leq 1.
		\end{equation}
		Using (\ref{eqn:3.2}), we have
		\[
		|g(s, p)| \leq \frac{1}{2}\left(1+\frac{|s^2-4p|}{4}\right) \leq 1 
		\]
		for every $(s, p) \in \Gamma$. Let $|g(s, p)|=1$ for some $(s, p) \in \Gamma$. Then 
		\begin{equation}\label{eqn:3.3}
			2=\left|1+\frac{s^2-4p}{4}\right| \quad \text{and so}, \quad \frac{s^2-4p}{4}=1
		\end{equation}
		which gives $s=0$ by the virtue of (\ref{eqn:3.2}). It follows from (\ref{eqn:3.3}) that $(s, p)=(0, -1)$. Therefore, $g$ peaks at $(0, -1)$ and consequently, $f=g \circ T_v$ peaks at $(s_0, p_0)$. Evidently, $f$ is a rational function that has no singularity in $\Gamma$. The proof is now complete.  
	\end{proof}
	
We now determine the distinguished boundary of a distinguished variety in $\G$, which is an analogue of Theorem \ref{thm803}. Obviously, one can imitate the proof of Theorem \ref{thm803} too here. However, we present an alternative proof leveraging the symmetrization map.
	
	\begin{thm}\label{dist_var}
		For a $\Gamma$-distinguished polynomial $p \in \C[z_1, z_2]$, the following are equivalent: 
		\begin{enumerate}
			\item $(s_0, p_0) \in Z(p) \cap b\Gamma$;
			\item $(s_0, p_0)$ is a peak point of $Z(p) \cap \Gamma$;
			\item $(s_0, p_0) \in b(Z(p) \cap \Gamma)$.
		\end{enumerate}
		Moreover,	$b(Z(p)\cap \Gamma)=Z(p) \cap b\Gamma$. 
	\end{thm}
	
	\begin{proof}  $(1) \Rightarrow (2)$ follows from Theorem \ref{thm::peak_Gamma} and  $ (2)\Rightarrow (3)$ follows from the definition of peak points. We show that $b(Z(p) \cap \Gamma) \subseteq Z(p) \cap b\Gamma$ proving $(3) \Rightarrow (1)$. Indeed,  we prove that $Z(p) \cap b\Gamma$ is a closed boundary of $X$. Let us define 
		$
		q(z_1, z_2)=p(z_1+z_2, z_1z_2)=p\circ \pi(z_1, z_2)$.	Following Example \ref{3.3} we conclude that $q$ is a toral polynomial. Take $g \in Rat(Z(p) \cap \Gamma)$ and define $f:=g \circ \pi$. Evidently, $f \in Rat(Z(q) \cap \DC^2)$. It follows from Theorem \ref{thm803} that there is $x \in Z(q) \cap \T^2$ such that 
		\begin{equation}\label{eqn803}
			|f(x)|=\|f\|_{\infty, Z(q) \cap \DC^2}.
		\end{equation}
		It is not difficult to see that $\pi(Z(q) \cap \DC^2)=Z(p) \cap \Gamma$. Thus
		\begin{equation*}
			\begin{split}
				|g\circ \pi(x)|
				=|f(x)|
				&=\sup\{|g\circ \pi(z_1, z_2)| \ : \ (z_1, z_2) \in Z(q) \cap \DC^2 \} \qquad [\text{By (\ref{eqn803})}]\\
				& =\sup\left\{|g\circ \pi(z_1, z_2)| \ : \ \pi(z_1, z_2) \in Z(p) \cap \Gamma \right\} \\
				& =\sup\left\{|g(\widetilde{s}, \widetilde{p})| \ : \ (\widetilde{s}, \widetilde{p}) \in Z(p) \cap \Gamma \right\} \\
				&=\|g\|_{\infty, Z(p) \cap \Gamma}.
			\end{split}
		\end{equation*}
		Note that $\pi(x) \in Z(p) \cap b\Gamma$, as $ x \in Z(q) \cap \T^2$. Therefore, $g$ attains its maximum modulus at a point in $Z(p) \cap b\Gamma$. Since $g \in Rat(Z(p) \cap \Gamma)$ is arbitrary, we have that $Z(p) \cap b\Gamma$ is a closed boundary for $Z(p) \cap \Gamma$. Consequently, $b(Z(p)\cap \Gamma) \subseteq Z(p) \cap b\Gamma$ and the proof is complete. 
	\end{proof}
	
	We cannot remove the hypothesis that $q$ is $\Gamma$-distinguished in the above theorem as the following example explains.

\begin{eg}
	The polynomial $q(z_1, z_2)=z_2$ is not $\Gamma$-distinguished. Let $(s_0, 0) \in Z(q) \cap \Gamma$. It was proved in \cite{AglerII16} that if $(s,p) \in \Gamma$, then $|s-\overline{s}p| \leq 1-|p|^2$. Thus, $(s_0,0) \in \Gamma$ implies that $|s_0| \leq 1$ and so, $Z(q) \cap \Gamma \subseteq \DC \times \{0\}$. For any $\lambda \in \DC \times \{0\}$, we have $(\lambda, 0)=\pi(\lambda, 0) \in \Gamma$. Therefore, $Z(q) \cap \Gamma=\DC \times \{0\}$.  Moreover, $Z(q) \cap b\Gamma=\emptyset$ as $|p_0|=1$ for $(s_0, p_0) \in b\Gamma$. The same computations as in Example \ref{eg_toral} give $Z(q) \cap b\Gamma \ne b(Z(q) \cap \Gamma)$.  \qed 
\end{eg}

\section{Dilation of toral contractions}\label{sec07}	
	
\noindent In this Section, we give a proof to Theorem \ref{main_toral} which provides necessary and sufficient conditions in order to answer the question raised in \cite{DasII}, if a toral pair of contractions dilates to a toral pair of isometries. It was shown in \cite{DasII} that every toral pair of isometries extends to a toral pair of unitaries. Thus, dilating a pair of commuting contractions to a toral pair of isometries is equivalent to dilating it to a toral pair of unitaries. Note that our aim is not to determine the success or failure of rational dilation on a distinguished variety in $\D^2$. Rather, our primary investigation is to find if dilation of a toral pair of contractions $(T_1,T_2)$ to a toral pair of isometries is equivalent to a particular distinguished variety in $\mathbb D^2$ being a complete spectral set for $(T_1,T_2)$. More precisely, we first settle the following question.
	
	\medskip

\noindent \textbf{Question.} Let $(T_1,T_2)$ be a toral pair of contractions that dilates to a toral pair of unitaries $(U_1,U_2)$. Let $p \in \C[z_1, z_2]$ be that toral polynomial for which $p(U_1, U_2)=0$. Is $Z(p) \cap \overline{\D}^2$ a complete spectral set for $(T_1,T_2)$, i.e., $\sigma_T(T_1,T_2) \subseteq Z(p) \cap \overline{\D}^2$ and 
\[
\left\|f(T_1,T_2)\right\| \leq \sup\left\{\left\|f(z_1, z_2)\right\| : (z_1, z_2) \in Z(p) \cap \overline{\D}^2 \right\}
\]
 for every matricial rational function $f=[f_{i,j}]$ with singularities off $Z(p) \cap \overline{\D}^2$ ?
 
 \medskip	
	
In this connection let us mention that a tuple $(S_1, \dotsc, S_n)$ of commuting operators on a Hilbert space $\mathcal{H}$ is said to have a \textit{commuting normal extension} or simply \textit{c.n.e.} if there is a tuple $(N_1, \dotsc, N_n)$ of commuting normal operators on some larger Hilbert space $\mathcal{K} \supseteq \mathcal{H}$ such that $\mathcal{H}$ is invariant under $N_1, \dotsc, N_n$ and $N_i|_\mathcal{H}=S_i$ for $i=1, \dotsc, n$. If we take 
	\[
	\mathcal{K}=\overline{span}\left\{N_1^{*m_1}\dotsc N_n^{*m_n}h \ : \ h \in \mathcal{H}, \ m_1, \dotsc, m_n \in \mathbb{N}\cup \{0\} \right\},
	\] 
	the minimal common reducing subspace of $N_1, \dotsc, N_n$ containing $\mathcal{H}$, then $(N_1, \dotsc, N_n)$ is unique up to unitary equivalence and is called the \textit{minimal commuting normal extension} or simply \textit{minimal c.n.e.} of $(S_1, \dotsc, S_n)$.

\begin{lem}[\cite{Lubin}, Corollary 2]\label{Lubin2}
	Let $\underline{S}=(S_1, \dotsc, S_n)$ be a commuting tuple of subnormal operators on a Hilbert space $\mathcal{H}$ and let $\underline{N}=(N_1, \dotsc, N_n)$ be the minimal c.n.e. of $S$. Then  $p(\underline{N})$ is unitarily equivalent to the minimal normal extension of $p(\underline{S})$ for any polynomial $p$ in $\C[z_1, \dots , z_n]$.
\end{lem}

 We start our campaign with Proposition 4.4 in \cite{DasII} which tells us that every toral pair of isometries extends to a toral pair of unitaries. We give an alternative proof to this result here.
	
	\begin{prop}[\cite{DasII}, Proposition 4.4]\label{toral_extend}
		Every toral pair of isometries extends to a toral pair of unitaries.
	\end{prop}
	
	\begin{proof}
		Let $(V_1, V_2)$ be a pair of commuting isometries on a Hilbert space $\HS$ and let $q$ be a toral polynomial such that $q(V_1, V_2)=0$. By Ando's theorem, there exists a pair of commuting unitaries $(\widetilde{U}_1, \widetilde{U}_2)$ acting on a Hilbert space $\widetilde{\mathcal{K}} \supseteq \HS$ such that $(V_1, V_2)=(\widetilde{U}_1|_\HS, \widetilde{U}_2|_\HS)$. Let 
		\[
		\mathcal{K}=\overline{\text{span}}\{\widetilde{U}_1^{*m}\widetilde{U}_2^{*n}h \ : \ h \in \HS \ \& \ m, n \in \N \cup \{0\} \}.
		\]
		Then, $\mathcal{K}$ is a closed subspace of $\widetilde{\mathcal{K}}$ that reduces $\widetilde{U}_1, \widetilde{U}_2$ and contains $\HS$. Define 
		$
		(U_1, U_2):=(\widetilde{U}_1|_\mathcal{K}, \widetilde{U}_2|_\mathcal{K}).
		$
This is a commuting pair of unitaries. Moreover, $(U_1, U_2)$ is the minimal c.n.e. of $(V_1, V_2)$. It follows from Lemma \ref{Lubin2} that $q(U_1, U_2)$ is unitarily equivalent to the minimal normal extension of $q(V_1, V_2)$. Thus $q(U_1, U_2)=0$ and the desired conclusion follows. 
	\end{proof}
	%---------------------------------------------------------------------------------
	As discussed earlier, for a commuting pair of Hilbert space contractions $(T_1, T_2)$ annihilated by a toral polynomial $q$, spectral mapping theorem gives that $\sigma_T(T_1, T_2) \subseteq Z(q) \cap \DC^2$. A natural question arises here if $Z(q) \cap \DC^2$ is a spectral set for $(T_1, T_2)$. We show by an example below that it is too much to expect for general contractions. However, we have an affirmative answer when $T_1, T_2$ are commuting isometries. To begin with, we have the following result for commuting unitaries.  

	\begin{prop}[\cite{DasII}, Proposition 4.5]\label{prop902}
		A pair of commuting unitaries $(U_1, U_2)$ is toral if and only if its joint spectrum is contained in a distinguished variety in $\D^2$. 
	\end{prop}
	
		By Proposition \ref{prop902}, a pair of commuting unitaries $(U_1, U_2)$ is annihilated by $q \in \C[z_1, z_2]$ if and only if $Z(q) \cap \T^2$ is a spectral set for $(U_1, U_2)$. The underlying reason is that $\sigma_T(\underline{N})$ is a spectral set for a commuting tuple $\underline{N}$ of normal operators.
	
		\begin{prop}
		Let $(V_1, V_2)$ be a  toral pair of isometries and let $q \in \C[z_1, z_2]$ be a toral polynomial. Then $q(V_1, V_2)=0$ if and only if $Z(q) \cap \DC^2$ is a spectral set for $(V_1, V_2)$.
	\end{prop}

	\begin{proof}
		It is evident that $q(V_1, V_2)=0$  if $Z(q) \cap \DC^2$ is a spectral set for $(V_1, V_2)$. Conversely, let $q(V_1, V_2)=0$. By spectral mapping theorem,  $\sigma_T(V_1, V_2) \subseteq Z(q) \cap \DC^2$. Following the proof of Proposition \ref{toral_extend}, we have that there is a pair of commuting unitaries $(U_1, U_2)$ acting on a Hilbert space $\mathcal{K} \supseteq \HS$ such that $(V_1, V_2)=(U_1|_\HS, U_2|_\HS)$ and $q(U_1, U_2)=0$. By Proposition \ref{prop902}, $Z(q) \cap \T^2$ is a spectral set for $(U_1, U_2)$. For any $ g \in Rat(Z(p) \cap \DC^2)$, we have 
		\[
		\|g(V_1, V_2)\|=\|g(U_1, U_2)|_\HS \| \leq \|g(U_1, U_2)\| \leq \|g\|_{\infty, \, Z(q) \cap \T^2} \leq \|g\|_{\infty, \, Z(q) \cap \DC^2}. 
		\]
		Therefore, $Z(p) \cap \DC^2$ is a spectral set for $(V_1, V_2)$. This finishes the proof.
	\end{proof}
	
		 Next, we show that for a pair $(T_1, T_2)$ of commuting contractions annihilated by a toral polynomial $q(z_1, z_2)$, the set $Z(q) \cap \DC^2$ need not be a spectral set for $(T_1, T_2)$. 
		
	\begin{eg}
		Let 
		\[
		T_1=\begin{pmatrix}
			0 & 0 & 0 \\
			1 & 0 & 0 \\
			0 & 1 & 0
		\end{pmatrix} \quad \text{and} \quad 
		T_2=\begin{pmatrix}
			0 & 0 & 0 \\
			0 & 0 & 0 \\
			-1 & 0 & 0
		\end{pmatrix}.
		\]
		It is not difficult to verify that $(T_1, T_2)$ is a commuting pair of contractions acting on $\C^3$. Moreover, $p(z_1, z_2)=(z_1-z_2)^3$ is a toral polynomial such that 
		\[
		p(T_1, T_2)=(T_1-T_2)^3=\begin{pmatrix}
			0 & 0 & 0 \\
			1 & 0 & 0 \\
			1 & 1 & 0
		\end{pmatrix}^3=\begin{pmatrix}
			0 & 0 & 0 \\
			0 & 0 & 0 \\
			0 & 0 & 0
		\end{pmatrix}.
		\]
		Let $g(z_1, z_2)=z_1-z_2$. Since $T_1-T_2 \ne 0$ and $Z(p)=Z(g)$, we have  
		$
		\|g(T_1, T_2)\| >0 =\|g\|_{\infty, Z(p) \cap \DC^2}$. Therefore, $Z(p) \cap \DC^2$ is not a spectral set for $(T_1, T_2)$ though $(T_1, T_2)$ is a toral pair of commuting contractions annihilated by $p(z_1, z_2)$.  
		\qed 
	\end{eg}
	
	\noindent The \textit{Neil parabola} 
	$
	W=\left\{(z_1, z_2) \in \D^2 : z_1^2=z_2^3\right\}
	$
	is an example of a distinguished variety in $\D^2$. Dritschel, Jury and McCullough proved (see Theorem 1.1 and Corollary 3.2 in \cite{Dritschel}) that rational dilation fails on the Neil parabola by assuring the existence of a contractive representation of $Rat(W)$ which is not completely contractive. Capitalizing their ideas, we produce our next example to show that there is a commuting pair of contractions $(X, Y)$ annihilated by an irreducible toral polynomial $q$ but $Z(q) \cap \DC^2$ is not a spectral set for $(X, Y)$.
	
		\begin{eg}
		Let $A(\D)$ be the algebra of functions that are continuous on $\DC$ and holomorphic in $\D$. Let $\mathcal{A}_0$ be the subalgebra of $A(\D)$ generated by the polynomials $z^2$ and $z^3$. Then 
		every element in $\mathcal{A}_0$ is of the form 
		\[
		g(z)=g_1(z)z^2+g_2(z)z^3
		\]
		for some $g_1, g_2 \in A(\D)$. Let $\mathcal{A}$ be the closure of $\mathcal{A}_0$ in $A(\D)$. A unital representation $\Phi : \mathcal{A} \to \mathcal{B}(\mathcal{H})$ is said to be \textit{contractive} if 
		$
		\|\Phi(f)\| \leq \|f\|_{\infty, \DC} $ for all $f \in \mathcal{A}$, where $\|f\|_{\infty, \DC}=\sup\{|f(z)| : z \in \DC\}$. It follows from Corollary 3.2 in \cite{Dritschel} that there exist commuting contractions $X$ and $Y$ acting on $\C^k$ (for some $k \in \N)$ with $X^3=Y^2$ such that the representation of $\mathcal{A}$ given by 
		\[
		\Phi: \mathcal{A} \to \mathcal{B}(\C^k), \quad \Phi(z^2)=X \quad \text{and} \quad \Phi(z^3)=Y
		\]
		is bounded but not contractive. Then there is some $f \in \mathcal{A}$ such that $\|\Phi(f)\|>\|f\|$. Since $\mathcal{A}=\overline{\mathcal{A}_0}$, there is a sequence $\{f_n\} \subset \mathcal{A}_0$ such that $\|f_n-f\|_{\infty, \DC} \to 0$ as $n \to \infty$. Using continuity of $\Phi$, it follows that there exists some $N \in \N$ such that $\|\Phi(f_N)\| > \|f_N\|_{\infty, \DC}$. By definition of $\mathcal{A}_0$, there exist $f_{N,1},\,f_{N, 2}$ in $A(\D)$ such that  
		$
		f_N(z)=f_{N,1}(z)z^2+f_{N, 2}(z)z^3.
		$ 
		Since polynomials are dense in $A(\D)$, one can obtain sequences $\{\alpha_{n}\}$ and $\{\beta_n\}$ of polynomials such that 
		\[
		\lim_{n \to \infty}\|f_{N, 1}-\alpha_n\|_{\infty, \DC}=0=	\lim_{n \to \infty}\|f_{N, 2}-\beta_n\|_{\infty, \DC}.
		\] 
		Consequently, we have that the sequence of polynomials $\{\gamma_n\}$ given by $\gamma_n(z)=\alpha_n(z)z^2+\beta_n(z)z^3$ converges to $f_N$ uniformly over $\DC$. Again by continuity of $\Phi$, it follows that there is some $N_0$ such that $\|\Phi(\gamma_{N_0})\|> \|\gamma_{N_0}\|$. Let us define $	p(z):=\gamma_{N_0}(z)$. Then there exist polynomials $p_1(z), p_2(z)$ such that
		$
		p(z)=p_1(z)z^2+p_2(z)z^3.
		$
		So, combining all facts together we have that 
		\begin{equation}\label{eg01}
			\|\Phi(p)\| >\|p\|_{\infty, \DC}.
		\end{equation}
		It follows from the definition of $p(z)$ that we can rewrite
		$
		p(z)=\widetilde{p}(z^2, z^3)
		$
		for some polynomial $\widetilde{p}$ (not necessarily unique) in two variables. Then it follows from the definition of the representation $\Phi$ that
		\begin{equation}\label{eg02}
			\Phi(p)=\widetilde{p}(X, Y).
		\end{equation}
		We show that 
		\begin{equation}\label{eg03}
			\|p\|_{\infty, \DC}=\|\widetilde{p}\|_{\infty, Z(q) \cap \DC^2},
		\end{equation}
		where $q(z_1, z_2)=z_1^3-z_2^2$. Note that there is some $\alpha \in \DC$ such that $\|p\|_{\infty, \DC}=|p(\alpha)|$ and thus,
		\[
		\|p\|_{\infty, \DC}=|p(\alpha)|=|\widetilde{p}(\alpha^2, \alpha^3)|\leq \sup\{|\widetilde{p}(z_1, z_2)| : z_1^3=z_2^2 \ \&  \ (z_1, z_2) \in \DC^2 \}= \|\widetilde{p}\|_{\infty, Z(q) \cap \DC^2}.
		\]
		Let $(\alpha, \beta) \in Z(q) \cap \DC^2$ such that $\|\widetilde{p}\|_{\infty, Z(q) \cap \DC^2}=|\widetilde{p}(\alpha, \beta)|$. To prove the reverse of inequality in (\ref{eg03}), we need to show that there is some $t_0 \in \DC$ such that $|\widetilde{p}(\alpha, \beta)|=|p(t_0)|$. Let $\alpha=re^{i\theta}$ for some $r \in [0,1]$ and $-\pi < \theta \leq \pi$. Define $t=\sqrt{r}e^{i\theta \slash 2}$ which is in $\DC$ and $t^2=\alpha$. For this choice of $t$, we have $\beta^2=\alpha^3=r^3e^{3i\theta}$. Thus either $\beta=r\sqrt{r}e^{3i\theta \slash 2}=t^3$ or $\beta=-r\sqrt{r}e^{3i\theta \slash 2}=-t^3$ and consequently, either $(\alpha, \beta)=(t^2, t^3)$ or $(\alpha, \beta)=((-t)^2, (-t)^3)$. In either case, we have that there is some $t_0 \in \DC$ such that $(\alpha, \beta)=(t_0^2, t_0^3)$. Consequently, 
		\[
		\|\widetilde{p}\|_{\infty, Z(q) \cap \DC^2}=|\widetilde{p}(\alpha, \beta)|=|\widetilde{p}(t_0^2, t_0^3)|=|p(t_0)| \leq \|p\|_{\infty, \DC}
		\]
		and thus (\ref{eg03}) holds. Therefore, it follows from (\ref{eg01})-(\ref{eg03}) that
		\begin{equation*}
			\begin{split}
				\|\widetilde{p}(X, Y)\|=\|\Phi(p)\|>\|p\|_{\infty, \DC} =\|\widetilde{p}\|_{\infty, Z(q) \cap \DC^2}.
			\end{split}
		\end{equation*} 
		Therefore, $Z(q) \cap \DC^2$ is not a spectral set for the commuting pair $(X, Y)$ of contractions but $q(z_1, z_2)$ is an irreducible toral polynomial such that $q(X, Y)=X^3-Y^2=0$. 
		\qed 
	\end{eg}
	
		Combining all facts together, we have the following result.
	
		\begin{lem}
		Let $(T_1, T_2)$ be a commuting pair of Hilbert space contractions and let $q(z_1, z_2)$ be a toral polynomial. If $Z(q) \cap \DC^2$ is a spectral set for $(T_1, T_2)$, then $q(T_1, T_2)=0$ but the converse does not hold even if $q$ is an irreducible polynomial. 
	\end{lem}
	
	Now we present an important theorem of this Section.
	
		\begin{thm}\label{Arveson_toral}
		Let $(T_1, T_2)$ be a pair of commuting contractions acting on a Hilbert space $\HS$. Then $(T_1, T_2)$ dilates to a toral pair of unitaries if and only if there is a toral polynomial $q(z_1, z_2)$ such that $Z(q) \cap \DC^2$ is a complete spectral set for $(T_1, T_2)$.
	\end{thm}
	
	\begin{proof}
		Suppose	$(T_1, T_2)$ admits dilation to a toral pair of unitaries $(U_1, U_2)$ on $\mathcal{K} \supseteq \mathcal{H}$. Let $q$ be a toral polynomial such that $q(U_1, U_2)=0$. It follows from Proposition \ref{prop902} and Theorem \ref{thm803} that $Z(q) \cap \T^2=b(Z(p) \cap \DC^2)$ is a spectral set for $(U_1, U_2)$. Also, for every $f \in Rat(\DC^2)$, we have 
		\begin{equation}\label{eqn9.04}
			f(T_1, T_2)=P_\mathcal{H}f(U_1, U_2)|_\mathcal{H}.
		\end{equation}
		In particular, $(\ref{eqn9.04})$ holds for every polynomial $f$ in two variables. By Proposition \ref{basicprop:01} and Lemma \ref{lem801}, $(\ref{eqn9.04})$ holds for every function $f$ in $Rat(Z(q) \cap \DC^2)$.  Consequently, there is a commuting pair of normal operators $(U_1, U_2)$ on $\mathcal{K} \supseteq \mathcal{H}$ with $\sigma_T(U_1, U_2)\subseteq b(Z(q)\cap \DC^2)$ such that 
		$
		f(T_1, T_2)=P_\mathcal{H}f(U_1, U_2)|_\mathcal{H}$ for every $f \in Rat(Z(q) \cap \DC^2)$. It follows from Theorem \ref{thm_Arveson} (Arveson's theorem) that $Z(q) \cap \DC^2$ is a complete spectral set for $(T_1, T_2)$.
		
		\smallskip
		
Conversely, assume that $Z(q) \cap \DC^2$ is a complete spectral set for $(T_1, T_2)$ for some toral polynomial $q(z_1, z_2)$. Again by Arveson's theorem, there is a commuting pair of normal operators $(U_1, U_2)$ on $\mathcal{K} \supseteq \mathcal{H}$ such that $\sigma_T(U_1, U_2)\subseteq b(Z(q)\cap \DC^2)$ and $f(T_1, T_2)=P_\mathcal{H}f(U_1, U_2)|_\mathcal{H}$ for every holomorphic polynomial $f$ in two variables. It follows from Theorem \ref{thm803} that $b(Z(q) \cap \DC^2)=Z(q) \cap \T^2$ and so, $\sigma_T(U_1, U_2) \subseteq Z(q) \cap \T^2$. By Proposition \ref{prop902}, $Z(q) \cap \T^2$ is a spectral set for $(U_1, U_2)$ and $q(U_1, U_2)=0$. Thus, $(U_1, U_2)$ is a toral pair of unitaries that dilates $(T_1, T_2)$ and we are done.
\end{proof}
	
		The following theorem is an obvious corollary of Proposition \ref{toral_extend} and Theorem \ref{Arveson_toral}.
	
		\begin{thm}\label{toral_isometries} Let $(V_1, V_2)$ be a pair of commuting isometries acting on a Hilbert space $\mathcal{H}$. Then the following are equivalent:
		\begin{enumerate}
			\item $(V_1, V_2)$ is a toral pair;
			\item $(V_1, V_2)$ has an extension to a toral pair of commuting unitaries;
			\item There is a toral polynomial $q$ such that $Z(q) \cap \DC^2$ is a spectral set for $(V_1, V_2)$;
			\item There is a toral polynomial $q$ such that $Z(q) \cap \DC^2$ is a complete spectral set for $(V_1, V_2)$.
		\end{enumerate}
	\end{thm}
	
	We now focus on proving Theorem \ref{main_toral}, one of the main results of this article. It demands a few preparatory results and they are provided below.
	
	\begin{prop}\label{prop9.08}
		If a pair of (toral) commuting contractions $(T_1, T_2)$ on a Hilbert space $\mathcal{H}$ admits a dilation to a (toral) pair of commuting isometries, then it has a minimal  dilation to a (toral) pair of commuting isometries.
	\end{prop}
	
	\begin{proof}
		Let $(\widetilde{V}_1, \widetilde{V}_2)$ on $\widetilde{\mathcal{K}} \supseteq \mathcal{H}$ be an isometric dilation of $(T_1, T_2)$. We can always have such a dilation by And\^{o}'s dilation theorem \cite{Ando}. Also, let $q$ be a toral polynomial such that $q(\widetilde{V}_1, \widetilde{V}_2)=0$. Define
		\[
		\mathcal{K}:= \overline{\text{span}}\{\widetilde{V}_1^i\widetilde{V}_2^jh \ : \ h \in \mathcal{H} \ \text{and} \ i, j \in \mathbb{N} \cup \{0\} \}.
		\] 
		It is easy to see that $\mathcal{K}$ is invariant under $\widetilde{V}_1^{i}$ and  $\widetilde{V}_2^j$ for any non-negative integers $i, j$. We denote by $(V_1,V_2)$ the following pair:
		$
		(V_1, V_2):=(\widetilde{V}_1|_\mathcal{K}, \widetilde{V}_2|_\mathcal{K}).
		$
		Then $(V_1, V_2)$ is also a commuting pair of isometries and $q(V_1, V_2)=0$. Moreover,
		\[
		\mathcal{K}= \overline{\text{span}}\{V_1^iV_2^jh \ : \ h \in \mathcal{H} \ \text{and} \ i, j \in \mathbb{N} \cup \{0\} \}.
		\] 
		Thus,
		$
		P_\mathcal{H}(V_1^iV_2^j)h=T_1^iT_2^jh	
		$
		for any $h \in \HS$ and $i, j \in \N \cup \{0\}$. Therefore, $(V_1, V_2)$ on $\mathcal{K}$ is a minimal isometric dilation of $(T_1, T_2)$ with $q(V_1, V_2)=0$. The proof is now complete.		
		\end{proof}
	
		\begin{prop}\label{prop9.09}
		Let $(T_1, T_2)$ be a commuting pair of contractions acting on a Hilbert space $\HS$ and let $(V_1, V_2)$ on $\mathcal{K} \supseteq \HS$ be a minimal isometric dilation of $(T_1,T_2)$. Then $(V_1^*, V_2^*)$ is a co-isometric extension of $(T_1^*, T_2^*)$. Conversely, if $(V_1^*, V_2^*)$ is a co-isometric extension of $(T_1^*, T_2^*)$, then $(V_1,V_2)$ is an isometric dilation of $(T_1,T_2)$.
	\end{prop}
	
	\begin{proof}
		
		We first prove that $T_1P_\mathcal{H}=P_\mathcal{H}V_1$ and $T_2P_\mathcal{H}=P_\mathcal{H}V_2$. By definition, we have that 
		\[
		\mathcal{K}=\overline{\text{span}}\{V_1^{i}V_2^{j}h \ : \ h \in \mathcal{H} \ \text{and} \ i, j \in \mathbb{N} \cup \{0\} \}.
		\]
		For $h \in \mathcal{H}$, we have
		$
		T_1P_\mathcal{H}(V_1^{i}V_2^{j}h)=T_1(T_1^{i}T_2^{j}h)=T_1^{i+1}T_2^{j}h=P_\mathcal{H}(V_1^{i+1}V_2^{j}h)=P_\mathcal{H}V_1(V_1^{i}V_2^{j}h)$. By continuity argument we have $T_1P_\mathcal{H}=P_\mathcal{H}V_1$ and similarly $T_2P_\mathcal{H}=P_\mathcal{H}V_2$. Also, for $h \in \mathcal{H}$ and $k \in \mathcal{K}$, we have
		\[
		\langle T_1^*h, k \rangle =\langle P_\mathcal{H}T_1^*h, k \rangle =\langle T_1^*h, P_\mathcal{H}k \rangle =\langle h, T_1P_\mathcal{H}k \rangle =\langle h, P_\mathcal{H}V_1k \rangle =\langle V_1^*h, k \rangle. 
		\]
		Hence, $T_1^*=V_1^*|_\mathcal{H}$ and similarly $T_2^*=V_2^*|_\mathcal{H}$. The converse part is obvious. 
	\end{proof}
	
		\begin{lem}\label{lem910}
		Let $\mathcal{H}_1$ and $\HS_2$ be Hilbert spaces. Let $V_1=\begin{bmatrix}
			T_1 & 0 \\
			C_1 & D_1
		\end{bmatrix}$ and $V_2=\begin{bmatrix}
			T_2 & 0 \\
			C_2 & D_2
		\end{bmatrix}$ be commuting contractions acting on $\mathcal{H}_1 \oplus \mathcal{H}_2$. If $f(T_1, T_2)=0=g(D_1, D_2)$ for $f, g \in \C[z_1, z_2]$, then $f(V_1, V_2)g(V_1, V_2)=0$. Moreover, $(V_1, V_2)$ is a toral pair if and only if both $(T_1, T_2)$ and $(D_1, D_2)$ are toral pairs. 
	\end{lem} 
	
	\begin{proof}
		For any $p \in \C[z_1, z_2]$, a routine calculation gives
		\begin{equation}\label{eqn_p(V)}
			p(V_1, V_2)=  \begin{bmatrix} 
				p(T_1, T_2) & 0  \\
				* & p(D_1, D_2) \\
			\end{bmatrix}. 
		\end{equation}
		Therefore, using the fact that $f(T_1, T_2)=0=g(D_1, D_2)$ we have
		\begin{equation*}
			\begin{split}
				f(V_1, V_2)g(V_1, V_2)
				&= g(V_1, V_2)f(V_1, V_2) \\
				&=\begin{bmatrix} 
					g(T_1, T_2) & 0  \\
					* & g(D_1, D_2) \\
				\end{bmatrix}
				\begin{bmatrix} 
					f(T_1, T_2) & 0  \\
					* & f(D_1, D_2) \\
				\end{bmatrix}\\
				&=\begin{bmatrix} 
					g(T_1, T_2) & 0  \\
					* & 0 \\
				\end{bmatrix}
				\begin{bmatrix} 
					0 & 0  \\
					* & f(D_1, D_2) \\
				\end{bmatrix}  =\begin{bmatrix} 
					0 & 0  \\
					0 & 0 \\
				\end{bmatrix}. \\
			\end{split}
		\end{equation*}
		Let $q$ be a toral polynomial such that $q(V_1, V_2)=0$. By (\ref{eqn_p(V)}), $q(T_1, T_2)=0$ and $q(D_1, D_2)=0$. Conversely, let $f$ and $g$ be toral polynomials such that $f(T_1, T_2)=0=g(D_1, D_2)$. Then, $f(V_1, V_2)g(V_1, V_2)=0$. Now $Z(fg)=Z(f) \cup Z(g)$ and so $fg$ is a toral polynomial. Thus, $(V_1, V_2)$ is a toral pair annihilated by $fg$ and the proof is complete.  
	\end{proof}

	Now we are in a position to present the proof of the main result of this Section. 
	
	\medskip	
	
\noindent \textbf{\textit{Proof of Theorem \ref{main_toral}}.} 
		$(1) \Rightarrow (2)$. Suppose $(T_1, T_2)$ dilates to a toral pair of commuting isometries $(\widetilde{V}_1, \widetilde{V}_2)$ on $\widetilde{\mathcal{K}} \supseteq \HS$. Let $q$ be a toral polynomial that annihilates $(\widetilde{V}_1, \widetilde{V}_2)$. Let us define 
		$
		(V_1, V_2):=(\widetilde{V}_1|_\mathcal{K}, \widetilde{V}_2|_\mathcal{K}),
		$
		where $
		\mathcal{K}= \overline{\text{span}}\{\widetilde{V}_1^i\widetilde{V}_2^jh \ : \ h \in \mathcal{H} \ \text{and} \ i, j \in \mathbb{N} \cup \{0\} \}
		$, the minimal isometric dilation space of $(T_1, T_2)$ with respect to the dilation $(\widetilde{V}_1, \widetilde{V}_2)$. It follows from Proposition \ref{prop9.08} that $(V_1, V_2)$ on $\mathcal{K}$ is an  isometric dilation of $(T_1, T_2)$ and $q(V_1, V_2)=0$. By Proposition \ref{prop9.09}, $(V_1^*, V_2^*)$ is a co-isometric extension of $(T_1^*, T_2^*)$. Let $\HS^\perp=\mathcal{K} \ominus \HS$. Then there exist  $C_1, C_2 \in \mathcal{B}(\HS, \HS^\perp)$ and $D_1, D_2 \in \mathcal{B}(\HS^\perp)$ such that
		\[
		V_1=\begin{bmatrix}
			T_1 & 0 \\
			C_1 & D_1
		\end{bmatrix} \quad \text{and} \quad 	V_2=\begin{bmatrix}
			T_2 & 0 \\
			C_2 & D_2
		\end{bmatrix}
		\]
		with respect to the orthogonal decomposition $\mathcal{K}=\HS \oplus \HS^{\perp}$. The fact that $(V_1, V_2)$ is a pair of commuting isometries gives the following: $(i) \, \, V_1V_2=V_2V_1$, $(ii)\, \, V_1^*V_1=I=V_2^*V_2$. Straightforward computations show that 
		\begin{equation}\label{eqn9.05}
			\begin{split}
				V_1V_2=V_2V_1 & \iff \begin{bmatrix}
					T_1 & 0 \\
					C_1 & D_1
				\end{bmatrix}\begin{bmatrix}
					T_2 & 0 \\
					C_2 & D_2
				\end{bmatrix}=\begin{bmatrix}
					T_2 & 0 \\
					C_2 & D_2
				\end{bmatrix}\begin{bmatrix}
					T_1 & 0 \\
					C_1 & D_1
				\end{bmatrix}\\
				& \iff \begin{bmatrix}
					T_1T_2 & 0 \\
					C_1T_2+D_1C_2 & D_1D_2
				\end{bmatrix}=\begin{bmatrix}
					T_2T_1 & 0 \\
					C_2T_1+D_2C_1 & D_2D_1
				\end{bmatrix}
			\end{split}
		\end{equation}
		and 
		\begin{equation}\label{eqn9.06}
			\begin{split}
				V_i^*V_i=I & \iff \begin{bmatrix}
					T_i^* & C_i^* \\
					0 & D_i^*
				\end{bmatrix}\begin{bmatrix}
					T_i & 0 \\
					C_i & D_i
				\end{bmatrix}=\begin{bmatrix}
					T_i^*T_i+C_i^*C_i & C_i^*D_i \\
					D_i^*C_i & D_i^*D_i
				\end{bmatrix}=\begin{bmatrix}
					I & 0 \\
					0 & I
				\end{bmatrix} \qquad (i=1,2).
			\end{split}
		\end{equation}
		It follows from (\ref{eqn9.05}) and (\ref{eqn9.06}) that $(D_1, D_2)$ is a commuting pair of isometries on $\HS^\perp$. Moreover, $q(D_1, D_2)=q(V_1, V_2)|_{\HS^\perp}=0$ and so, $(D_1, D_2)$ is a toral pair. Again by (\ref{eqn9.05}) and (\ref{eqn9.06}), the operator equations in condition-$(2)$ of this theorem follow immediately.
		
		\medskip
		
\noindent $(2) \Rightarrow (1)$.	Suppose there exist a Hilbert space $\mathcal{K} \supseteq \HS$, a toral pair of isometries $(D_1, D_2)$ on $\HS^\perp=\mathcal{K}\ominus \HS$ and $C_1, C_2 \in \mathcal{B}(\HS, \HS^\perp)$ such that the operator equations given in condition-$(2)$ hold. Set 
		\[
		V_1=\begin{bmatrix}
			T_1 & 0 \\
			C_1 & D_1
		\end{bmatrix} \quad \text{and} \quad 	V_2=\begin{bmatrix}
			T_2 & 0 \\
			C_2 & D_2
		\end{bmatrix}
		\]
		with respect to orthogonal decomposition $\mathcal{K}=\HS \oplus \HS^{\perp}$. In view of the operator equations in condition-$(2)$, it is evident from (\ref{eqn9.05}) and (\ref{eqn9.06}) that $(V_1, V_2)$ is a pair of commuting isometries.  Evidently, $T_1^*=V_1^*|_\mathcal{H}, T_2^*=V_2^*|_\mathcal{H}$ and hence $(V_1, V_2)$ dilates $(T_1, T_2)$. It remains to prove that $(V_1, V_2)$ is a toral pair. Since $(T_1, T_2)$ and $(D_1, D_2)$ are toral pairs, there are toral polynomials $f$ and $g$ such that $f(T_1, T_2)=0=g(D_1, D_2)$. It follows from Lemma \ref{lem910} that the toral polynomial $fg$ annihilates $(V_1, V_2)$ and so, $(V_1, V_2)$ is a toral pair.
		
		\medskip
		
		\noindent $(1) \Leftrightarrow (3)$. This follows from Theorems \ref{Arveson_toral} \& \ref{toral_isometries}.	
		
\medskip		
		
Moreover, in view of Proposition \ref{toral_extend} it is clear from the proof of $(2) \Rightarrow (1)$ that $(V_1, V_2)$ extends a pair of commuting unitaries $(U_1, U_2)$ such that $f(U_1, U_2)g(U_1, U_2)=0$. So, $(T_1, T_2)$ dilates to $(U_1, U_2)$. It follows from Theorem \ref{Arveson_toral} that $Z(fg) \cap \DC^2$ is a complete spectral set for $(T_1, T_2)$. The proof is now complete. \qed

		\section{Dilation of $\Gamma$-distinguished $\Gamma$-contractions}\label{sec08}

	%\vspace{0.2cm}
	
	 \noindent In this Section, we first find analogues of the results of Section \ref{sec07} in the symmetrized bidisc setting. Then we investigate more about dilation of a $\Gamma$-distinguished $\Gamma$-contraction $(S,P)$ when the defect spaces $\mathcal D_P, \, \mathcal D_{P^*}$ are finite dimensional. We begin with the most expected and natural lemma that guarantees that every $\Gamma$-distinguished $\Gamma$-isometry extends to a $\Gamma$-distinguished $\Gamma$-unitary. 
	
		\begin{prop}\label{prop8.1}
		Every $\Gamma$-distinguished $\Gamma$-isometry admits a $\Gamma$-distinguished $\Gamma$-unitary extension.
	\end{prop}	
	
		\begin{proof}
		Let $(T,V)$ be a $\Gamma$-isometry acting on a Hilbert space $\mathcal{H}$ and $f(T, V)=0$ for some $\Gamma$-distinguished polynomial $f(z_1, z_2)$. Let $(N_1, N_2)$ be a $\Gamma$-unitary extension of $(T, V)$ and also let, 
		\[
		\mathcal{K}=\overline{span}\left\{N_1^{*i}N_2^{*j}h \ : \ h \in \mathcal{H}, \  i, j \in \mathbb{N}\cup \{0\} \right\}.
		\]		
		We now prove that the minimal c.n.e. $(U_1, U_2)=(N_1|_\mathcal{K}, N_2|_\mathcal{K})$ of $(T, V)$ is a $\Gamma$-distinguished $\Gamma$-unitary. Since $(N_1, N_2)$ is a $\Gamma$-unitary and $\mathcal{K}$ is a common reducing subspace of $N_1, N_2$, thus, $(U_1, U_2)$ is a $\Gamma$-unitary on $\mathcal{K}$. 
		%Let $p(z,w)$ be a holomorphic polynomial in $2$-variables. Then $p(T, V)$ is a subnormal operator and let $N(p)$ be the minimal normal extension of $p(T, V)$. Bram proved \cite{Bram} that a subnormal operator $S$ satisfies the spectral inclusion relation 
		%\[
		%\sigma(N) \subseteq \sigma(S)
		%\]
		%where $N$ is the minimal normal extension of $N$. A result due to Stampfli yields that the norm of a subnormal operator equals its spectral radius. Putting everything together, we have that 
		%\[
		%\|N(p)\|=\max\{|\lambda| \ : \ \lambda \in \sigma(N_p)\} \leq  \max\{|\lambda| \ : \ \lambda \in \sigma(p(T, V))\}=\|p(T, V)\|.
		%\]
		It follows from Lemma \ref{Lubin2} that $f(U_1, U_2)$ is unitarily equivalent to the minimal normal extension of $f(T, V)$ and thus, $f(U_1, U_2)=0$. Therefore, $(U_1, U_2)$ is a $\Gamma$-distinguished $\Gamma$-unitary on $\mathcal{K}$ so that $U_1|_\mathcal{H}=T$ and $U_2|_\mathcal{H}=V$.	
\end{proof}
	
	After the above theorem, it is evident that dilation of a $\Gamma$-contraction to a $\Gamma$-distinguished $\Gamma$-isometry	 implies and is implied by dilation to a $\Gamma$-distinguished $\Gamma$-unitary. The next corollary is an easy consequence of the above proposition.
	
		\begin{cor}
		A $\Gamma$-isometry is $\Gamma$-distinguished if and only if it extends to a $\Gamma$-distinguished $\Gamma$-unitary.
	\end{cor}
	
		We look a little deeper into the nature of a $\Gamma$-distinguished $\Gamma$-unitary and have the following.
		 
		\begin{prop}\label{G_distt._G_unitary}
		Let $(U_1, U_2)$ be a  $\Gamma$-unitary and let $p \in \C[z_1, z_2]$. Then $p(U_1, U_2)=0$ if and only if $Z(p) \cap b\Gamma$ is a spectral set for $(U_1, U_2)$.
	\end{prop}
	
		\begin{proof}
		Suppose that a polynomial $p$ annihilates a $\Gamma$-unitary $(U_1, U_2)$. Then the spectral mapping theorem implies that
		$p(\sigma_T(U_1, U_2))=\sigma(p(U_1, U_2))=\{0\}$. Thus, $\sigma_T(U_1, U_2) \subseteq Z(p) \cap b\Gamma$. Conversely, let us assume that $p(\sigma_T(U_1, U_2))=\{0\}$. Since $U_1$ and $U_2$ are commuting normal operators, $p(U_1, U_2)$ is a normal operator too and we have 
		\[
		\|p(U_1, U_2)\|=\sup\{|\lambda|: \lambda \in \sigma(p(U_1, U_2))\}=\sup\{|\lambda|: \lambda \in p(\sigma_T(U_1, U_2))\}=0.
		\]
		Therefore, $p$ annihilates $(U_1, U_2)$.
	\end{proof}
	
	It follows from the above proposition that a $\Gamma$-unitary $(U_1, U_2)$ is annihilated by a $\Gamma$-distinguished polynomial $p$ if and only if $\sigma_T(U_1, U_2) \subseteq Z(p)\cap b\Gamma$. The latter is possible if and only if  $Z(p) \cap b\Gamma$ is a spectral set for $(U_1, U_2)$. We prove an analogue of the above result for the $\Gamma$-isometries. If $(V_1, V_2)$ is a $\Gamma$-isometry on a Hilbert space $\mathcal{H}$ then it has a $\Gamma$-unitary extension $(U_1, U_2)$ on some space $\mathcal{K} \supseteq \mathcal{H}$. For any polynomial $p(z_1, z_2)$, it follows that $p(V_1, V_2)=p(U_1, U_2)|_{\mathcal{H}}$. This shows that $p(V_1, V_2)$ is a subnormal operator and hence a hyponormal operator. It is a well-known result \cite{Stampfli} that for a hyponormal operator $T$, $\|T\|$ is the spectral radius of $T$. Thus, $\|p(V_1, V_2)\|$ is equal to the spectral radius of $p(V_1, V_2)$ and we can easily prove the following result.
	
		\begin{prop}
		A $\Gamma$-isometry $(V_1, V_2)$ is $\Gamma$-distinguished if and only if  there is a $\Gamma$-distinguished polynomial $p$ such that $Z(p) \cap \Gamma$ is a spectral set for $(V_1, V_2)$.
	\end{prop}
	
		\begin{proof}
		Let $(V_1, V_2)$ be  annihilated by a $\Gamma$-distinguished polynomial $p$. Then $\sigma_T(V_1, V_2) \subseteq Z(p) \cap \Gamma$. Since $\sigma_T(V_1, V_2)$ is a spectral set for $(V_1, V_2)$, we have that $Z(p) \cap \Gamma$ is a spectral set for $(V_1, V_2)$. Conversely, if there is a $\Gamma$-distinguished polynomial $p$ such that $Z(p) \cap \Gamma$ is a spectral set for $(V_1, V_2)$, then $\sigma(p(V_1, V_2))=\{0\}$. Since $V_1$ and $V_2$ have commuting normal extensions, it follows that $p(V_1, V_2)$ is a subnormal operator. Hence, 
		$
		\|p(V_1, V_2)\|= \sup\{|z| \; : \; z \in \sigma(p(V_1, V_2))\}=0. 
		$
		This gives that $p(V_1, V_2)=0$.
	\end{proof}
	
	Now, we give a counter-example to show that the above results are not true in general for a $\Gamma$-contraction. Indeed, we show that there is a $\Gamma$-contraction $(S,P)$ that is annihilated by a $\Gamma$-distinguished polynomial $q$ but $Z(q) \cap \Gamma$ is not a spectral set for $(S,P)$.
	
		\begin{eg}
		Consider the commuting operators acting on $\mathbb{C}^4$ given by
		\[A_1=\begin{pmatrix}
			0 & 0 & 0 & 0\\
			1 & 0 & 0 & 0\\
			0 & 1 & 0 & 0\\
			1 & 0 & 1 & 0\\
		\end{pmatrix}	\quad \text{and} \quad   A_2=\begin{pmatrix}
			0 & 0 & 0 & 0\\
			0 & 0 & 0 & 0\\
			-1 & 0 & 0 & 0\\
			0 & -1 & 0 & 0\\
		\end{pmatrix}.
		\]
		We define $T_1=rA_1$ and $T_2=rA_2$, where $r \ne 0$ is chosen in such a way that $\|T_1\|, \|T_2\| \leq 1$. The commuting pair of operators $(S,P)=(T_1+T_2, T_1T_2)$ is a $\Gamma$-contraction and is given by 
		\[
		S=r\begin{pmatrix}
			0 & 0 & 0 & 0\\
			1 & 0 & 0 & 0\\
			-1 & 1 & 0 & 0\\
			1 & -1 & 1 & 0\\
		\end{pmatrix}, \quad P=r^2\begin{pmatrix}
			0 & 0 & 0 & 0\\
			0 & 0 & 0 & 0\\
			0 & 0 & 0 & 0\\
			-1 & 0 & 0 & 0\\
		\end{pmatrix}.
		\]
		We have
		\[
		4P-S^2=4r^2\begin{pmatrix}
			0 & 0 & 0 & 0\\
			0 & 0 & 0 & 0\\
			0 & 0 & 0 & 0\\
			-1 & 0 & 0 & 0\\
		\end{pmatrix}-r^2\begin{pmatrix}
			0 & 0 & 0 & 0\\
			0 & 0 & 0 & 0\\
			1 & 0 & 0 & 0\\
			-2 & 1 & 0 & 0\\
		\end{pmatrix}=r^2\begin{pmatrix}
			0 & 0 & 0 & 0\\
			0 & 0 & 0 & 0\\
			-1 & 0 & 0 & 0\\
			-2 & -1 & 0 & 0\\
		\end{pmatrix}.
		\] 
		This shows that the $\Gamma$-distinguished polynomial $p(z_1, z_2)=(4z_2-z_1^2)^2$ annihilates $(S,P)$. Let if possible, $Z(p) \cap \Gamma$ be a spectral set for $(S,P)$ then we must have that 
		\[
		\|f(S,P)\|\leq \|f\|_{\infty, Z(p)\cap \Gamma}=\sup \{|f(z_1, z_2)|: (z_1, z_2) \in Z(p) \cap \Gamma \}
		\] 
		for every $f \in \mathbb{C}[z_1, z_2]$. In particular, the von Neumann's inequality must hold for $f(z_1, z_2)=4z_2-z_1^2$.  Since $Z(f)=Z(p)$, we have $\|f\|_{\infty, Z(p)\cap \Gamma}=0$ and $\|f(S,P)\|=\|4P-S^2\| > 0$. Hence, $Z(p) \cap \Gamma$ is a not a spectral set for $(S,P)$. 
		\qed 
	\end{eg}
	
	The following is an analogue of Theorem \ref{Arveson_toral} and is a main result of this Section.
	
		\begin{thm}\label{Main}
		Let $(S,P)$ be a $\Gamma$-contraction on a Hilbert space $\mathcal{H}$. Then $(S,P)$ admits a $\Gamma$-distinguished $\Gamma$-unitary dilation if and only if there is a $\Gamma$-distinguished polynomial $p(z_1, z_2)$ such that $Z(p) \cap \Gamma$ is a complete spectral set for $(S,P)$.
	\end{thm}
	
		\begin{proof}
		Assume that	$(S,P)$ admits a $\Gamma$-unitary dilation $(U_1, U_2)$ on $\mathcal{K} \supseteq \mathcal{H}$ such that $p(U_1, U_2)=0$ for some $\Gamma$-distinguished polynomial $p$. It follows from Proposition \ref{G_distt._G_unitary} and Theorem \ref{dist_var} that $b(Z(p) \cap \Gamma)=Z(p) \cap b\Gamma$ and that $b(Z(p) \cap \Gamma)$ is a spectral set for $(U_1, U_2)$. Furthermore, for every $f \in Rat(\Gamma)$, we have 
		\begin{equation}\label{eq:9.1}
			f(S, P)=P_\mathcal{H}f(U_1, U_2)|_\mathcal{H}.
		\end{equation}
		In particular, $(\ref{eq:9.1})$ holds for every polynomial $f$ in two variables. By Proposition \ref{basicprop:01} and Lemma \ref{lem801}, $(\ref{eq:9.1})$ holds for every function $f$ in $Rat(Z(p) \cap \Gamma)$.  Consequently, there is a commuting pair of normal operators $(U_1, U_2)$ on $\mathcal{K} \supseteq \mathcal{H}$ with $\sigma_T(U_1, U_2)\subseteq b(Z(p)\cap \Gamma)$ such that 
		$
		f(S, P)=P_\mathcal{H}f(U_1, U_2)|_\mathcal{H}
		$
		for every $f \in Rat(Z(p) \cap \Gamma)$. By Theorem \ref{thm_Arveson} (Arveson's theorem), $Z(p) \cap \Gamma$ is a complete spectral set for $(S, P)$.
		
		\smallskip
		
		Conversely, assume that $Z(p) \cap \Gamma$ is a complete spectral set for $(S,P)$ for some $\Gamma$-distinguished polynomial $p$. Again by Arveson's theorem, there is a commuting pair of normal operators $(U_1, U_2)$ on $\mathcal{K} \supseteq \mathcal{H}$ such that $\sigma_T(U_1, U_2)\subseteq b(Z(p)\cap \Gamma)$ and $f(S, P)=P_\mathcal{H}f(U_1, U_2)|_\mathcal{H}$ for every holomorphic polynomial $f$ in two variables. It follows from Theorem \ref{dist_var} that $b(Z(p) \cap \Gamma)=Z(p) \cap b\Gamma$ and so, $\sigma_T(U_1, U_2) \subseteq Z(p) \cap b\Gamma$. So, by Proposition \ref{G_distt._G_unitary}, $(U_1, U_2)$ is a $\Gamma$-unitary and $p(U_1, U_2)=0$. Hence, $(U_1, U_2)$ is a $\Gamma$-distinguished $\Gamma$-unitary that dilates $(S,P)$. The proof is now complete.		
	\end{proof}

The following theorem follows as an immediate corollary of Proposition \ref{prop8.1} and Theorem \ref{Main}.
	
\begin{thm}\label{G_isometries}
		 Let $(V_1, V_2)$ be a $\Gamma$-isometry on a space $\mathcal{H}$. Then the following are equivalent:
		\begin{enumerate}
			\item $(V_1, V_2)$ is $\Gamma$-distinguished;
			\item $(V_1, V_2)$ has a $\Gamma$-distinguished $\Gamma$-unitary extension;
			\item There is a $\Gamma$-distinguished polynomial $p$ such that $Z(p) \cap \Gamma$ is a spectral set for $(V_1, V_2)$;
			\item There is a $\Gamma$-distinguished polynomial $p$ such that $Z(p) \cap \Gamma$ is a complete spectral set for $(V_1, V_2)$.
		\end{enumerate}
	\end{thm}
	%----------------------------------------------------

	We now present a few examples of $\Gamma$-distinguished $\Gamma$-contractions that admit a $\Gamma$-distinguished $\Gamma$-isometric dilation. Recall from the literature (e.g. see \cite{Agler-Young-V}) that the \textit{royal variety} in the symmetrized bidisc is defined to be the set 
	$
	R:=\left\{(2z, z^2) \ : \ z \in \D \right\},
	$
	which is a distinguished variety in the symmetrized bidisc.
	
	\begin{eg}\label{eg811}
		Let $f(z_1, z_2)=z_1^2-4z_2$ which is a $\Gamma$-distinguished polynomial. Then 
		\[
		Z(f) \cap \Gamma=\{(z_1, z_2) \in \Gamma \ : \ z_1^2=4z_2 \}=\{(2z, z^2) \ : \ z \in \DC\}=\overline{R}.
		\]	
		Let $(S, P)$ be a $\Gamma$-contraction such that $f(S, P)=0$ and so, $P=S^2\slash 4$. Take any matricial polynomial $[f_{ij}]_{1 \leq i, j \leq n}$ and define $g_{ij}(z)=f_{ij}(2z, z^2)$ for every $i, j$. Then		
		\begin{equation*}
			\begin{split}
				\|[f_{ij}(S, P)]_{i, j}\| 
				=    \|[f_{ij}(S, S^2\slash 4)]_{i, j}\|
				& =    \|[g_{ij}(S\slash 2)]_{i, j}\| 
				\\ & \leq   \max_{z \in \DC}  \|[g_{ij}(z)]_{i, j}\| \\
				& \leq   \max \{ \|[f_{ij}(2z, z^2)]_{i, j}\| \ : \ z \in \DC \}\\
				& \leq   \max \{ \|[f_{ij}(z_1, z_2)]_{i, j}\| \ : \ (z_1, z_2) \in Z(f) \cap \Gamma \}.
\end{split}
\end{equation*}
Thus, $Z(f) \cap \Gamma$ is a complete spectral set for $(S, P)$. It follows from Theorem \ref{Main}
		that $(S, P)$ dilates to a $\Gamma$-distinguished $\Gamma$-isometry. \qed 
		
	\end{eg}

The first named author of this article and Shalit gave an explicit description of a distinguished variety in the symmetrized bidisc in \cite{PalShalit1}. We recall that result for our purpose.
\begin{thm} [\cite{PalShalit1}, Theorem 3.5]\label{PalShalit}
		Let $A$ be a square matrix with $\omega(A) <1,$ and let $W$ be the subset of $\mathbb{G}_2$ defined by 
		\[
		W=\{(z_1, z_2)\in \mathbb{G}_2 \ :\ det(A+ z_2A^*-z_1I)=0\}.
		\]
		Then $W$ is a distinguished variety. Conversely, every distinguished variety in $\mathbb{G}_2$ has the form $\{(z_1, z_2)\in \mathbb{G}_2 \ :\ det(A+ z_2A^*-z_1I)=0\},$ for some matrix $A$ with $\omega(A) \leq 1$.
	\end{thm}   
	\begin{eg}
		For $a \in \D$, define $f(z_1, z_2):=z_1-\overline{a}z_2-a$. It follows from Theorem \ref{PalShalit} that $f$ is a $\Gamma$-distinguished polynomial.  Let $(S, P)$ be a $\Gamma$-contraction such that $f(S, P)=0$ and so, $S=\overline{a}P+aI$. Observe that if $z \in \DC$, then $(s, p)=(\overline{a}z+a, z) \in \Gamma$ as $|s| \leq 2$ and 
		\[
		|s-\overline{s}p|=|\overline{a}z+a-(a\overline{z}+\overline{a})z|=|a(1-|z|^2)|< 1-|z|^2=  1-|p|^2.
		\]
		Take any matricial polynomial $[f_{ij}]_{1 \leq i, j \leq n}$ and define $h_{ij}(z)=f_{ij}(\overline{a}z+a, z)$ for every $i, j$. Then
		
		\begin{equation*}
			\begin{split}
				\|[f_{ij}(S, P)]_{i, j}\| 
				=    \|[f_{ij}(\overline{a}P+aI, P)]_{i, j}\|
				& =    \|[h_{ij}(P)]_{i, j}\| 
				\\ & \leq   \max_{z \in \DC}  \|[h_{ij}(z)]_{i, j}\| \\
				& \leq   \max \{ \|[f_{ij}(\overline{a}z+a, z)]_{i, j}\| \ : \ z \in \DC \}\\
				& \leq   \max \{ \|[f_{ij}(z_1, z_2)]_{i, j}\| \ : \ (z_1, z_2) \in Z(f) \cap \Gamma \}.
			\end{split}
		\end{equation*}
		Therefore, $Z(f) \cap \Gamma$ is a complete spectral set for $(S, P)$. By Theorem \ref{Main}, 
		$(S, P)$ dilates to a $\Gamma$-distinguished $\Gamma$-isometry. 
		\qed
	\end{eg}
	
	We now move to prepare for giving a proof to Theorem \ref{thm814}, the main result of this Section and an analogue of Theorem \ref{main_toral}. Needless to mention, we shall develop similar preparatory results as in the previous section. So, let us begin with the following proposition.	
	
	\begin{prop}\label{prop1012}
		Let $(S, P)$ be a $\Gamma$-distinguished $\Gamma$-contraction acting on a Hilbert space $\HS$. If $(S, P)$ has a $\Gamma$-distinguished $\Gamma$-isometric dilation, then it has a minimal $\Gamma$-distinguished $\Gamma$-isometric dilation.
	\end{prop}
	
	\begin{proof}
		Let $(\widetilde{T}, \widetilde{V})$ be a $\Gamma$-distinguished $\Gamma$-isometry acting on a Hilbert space $\widetilde{\mathcal{K}} \supseteq \HS$ such that $(\widetilde{T}, \widetilde{V})$ dilates $(S, P)$. Let $q$ be a $\Gamma$-distinguished polynomial that annihilates $(\widetilde{T}, \widetilde{V})$. Let us consider
		\[
		\mathcal{K}= \overline{\text{span}}\{\widetilde{T}^i\widetilde{V}^jh \ : \ h \in \mathcal{H} \ \text{and} \ i, j \in \mathbb{N} \cup \{0\} \}.
		\] 
		Evidently, $\mathcal{K}$ is invariant under $\widetilde{V}_1^{i}$ and  $\widetilde{V}_2^j$ for any non-negative integers $i, j$. We denote by
		$
		(T, V)$ the pair $(\widetilde{T}|_\mathcal{K}, \widetilde{V}|_\mathcal{K})
		$
		and so, $(T, V)$ is a $\Gamma$-isometry and $q(T, V)=0$. Moreover,
		\[
		\mathcal{K}= \overline{\text{span}}\{T^iV^jh \ : \ h \in \mathcal{H} \ \text{and} \ i, j \in \mathbb{N} \cup \{0\} \}.
		\] 
		Thus 
		$
		P_\mathcal{H}(T^iV^j)h=S^iP^jh	
		$
		for any $h \in \HS$ and $i, j \in \N \cup \{0\}$. Therefore, $(T, V)$ on $\mathcal{K}$ is a minimal isometric dilation of $(S, P)$ with $q(T, V)=0$ and the proof is complete.
	\end{proof}
	%-------------------------------------------------------------------------------
	\begin{prop}\label{prop1013}
		Let $(S, P)$ be a $\Gamma$-contraction acting on a Hilbert space $\mathcal{K}$ and let $(T, V)$ be its minimal $\Gamma$-isometric dilation acting on a Hilbert space $\mathcal{K}$. Then $(T^*, V^*)$ is a $\Gamma$-co-isometric extension of $(S^*, P^*)$. Conversely, if $(T^*, V^*)$ is a $\Gamma$-co-isometric extension of $(S^*, P^*)$ then $(T,V)$ is a $\Gamma$-isometric dilation of $(S,P)$.
	\end{prop}
	
	\begin{proof}
		It follows from the definition of minimality that $\mathcal{K}= \overline{\text{span}}\{T^iV^jh \ : \ h \in \mathcal{H} \ \text{and} \ i, j \in \mathbb{N} \cup \{0\} \}$. We show that $SP_\HS=P_\HS T$ and $PP_\HS=P_\HS V$. For any $h \in \HS$, we have
		$
		SP_\mathcal{H}(T^{i}V^{j}h)=S(S^{i}P^{j}h)=S^{i+1}P^{j}h=P_\mathcal{H}(T^{i+1}V^{j}h)=P_\mathcal{H}T(T^{i}V^{j}h)$.	Using continuity argument, it follows that $SP_\mathcal{H}=P_\mathcal{H}T$ and similarly, $PP_\mathcal{H}=P_\mathcal{H}V$. Also for $h \in \mathcal{H}$ and $k \in \mathcal{K}$, we have
		\[
		\langle S^*h, k \rangle =\langle P_\mathcal{H}S^*h, k \rangle =\langle S^*h, P_\mathcal{H}k \rangle =\langle h, SP_\mathcal{H}k \rangle =\langle h, P_\mathcal{H}Tk \rangle =\langle T^*h, k \rangle. 
		\]
		Hence, $S^*=T^*|_\mathcal{H}$ and similarly $P^*=V^*|_\mathcal{H}$. The converse part is obvious.		 
	\end{proof}
	
	\begin{lem}\label{lem1014}
		Let $\mathcal{H}_1$ and $\HS_2$ be Hilbert spaces. Let $V_1=\begin{bmatrix}
			T_1 & 0 \\
			C_1 & D_1
			
		\end{bmatrix}$ and $V_2=\begin{bmatrix}
			T_2 & 0 \\
			C_2 & D_2
		\end{bmatrix}$ be commuting operators acting on $\mathcal{H}_1 \oplus \mathcal{H}_2$. If $f(T_1, T_2)=0=g(D_1, D_2)$ for $f, g \in \C[z_1, z_2]$, then $f(V_1, V_2)g(V_1, V_2)=0$. Moreover, $(V_1, V_2)$ is a $\Gamma$-distinguished $\Gamma$-contraction if and only if both $(T_1, T_2)$ and $(D_1, D_2)$ are $\Gamma$-distinguished $\Gamma$-contractions. 
	\end{lem}
	
	\begin{proof}
The proof is similar to that of Lemma \ref{lem910} and thus we skip it.
\end{proof}	
	
Having in hand all preparatory results, we now give a proof to the main result of this Section.

\medskip	
	
\noindent \textit{\textbf{Proof of Theorem \ref{thm814}.}}			
$(1)\Rightarrow (2)$. Suppose $(S, P)$ dilates to a $\Gamma$-distinguished $\Gamma$-isometry $(\widetilde{T}, \widetilde{V})$ acting on a space $\widetilde{K} \supseteq \HS$. Then there is a $\Gamma$-distinguished polynomial $p \in \C[z_1, z_2]$ such that $p(\widetilde{T}, \widetilde{V})=0$. Let us denote by	$
		(T, V)$ the pair $(\widetilde{T}|_\mathcal{K}, \widetilde{V}|_\mathcal{K})$, where $
\mathcal{K}= \overline{\text{span}}\{\widetilde{T}^i\widetilde{V}^jh \ : \ h \in \mathcal{H} \ \text{and} \ i, j \in \mathbb{N} \cup \{0\} \}$, the minimal isometric dilation space of $(S, P)$ with respect to the dilation $(\widetilde{T}, \widetilde{V})$. It follows from Proposition \ref{prop1012} that $(T, V)$ on $\mathcal{K}$ is a $\Gamma$-isometric dilation of $(S, P)$ and $p(T, V)=0$. By Proposition \ref{prop1013}, $(T^*, V^*)$ is a $\Gamma$-co-isometric extension of $(S^*, P^*)$. Let $\HS^\perp=\mathcal{K} \ominus \HS$. Then there exist  $C_1, C_2 \in \mathcal{B}(\HS, \HS^\perp)$ and $D_1, D_2 \in \mathcal{B}(\HS^\perp)$ such that
		\[
		T=\begin{bmatrix}
			S & 0 \\
			C_1 & D_1
		\end{bmatrix} \quad \text{and} \quad 	V=\begin{bmatrix}
			P & 0 \\
			C_2 & D_2
		\end{bmatrix}
		\]
		with respect to the orthogonal decomposition $\mathcal{K}=\HS \oplus \HS^{\perp}$. A straightforward computation shows that $TV=VT$ if and only if
		\begin{equation*}
			\begin{split}
				\begin{bmatrix}
					S & 0 \\
					C_1 & D_1
				\end{bmatrix}\begin{bmatrix}
					P & 0 \\
					C_2 & D_2
				\end{bmatrix}=\begin{bmatrix}
					P & 0 \\
					C_2 & D_2
				\end{bmatrix}\begin{bmatrix}
					S & 0 \\
					C_1 & D_1
				\end{bmatrix} 
				& \iff \begin{bmatrix}
					SP & 0 \\
					C_1P+D_1C_2 & D_1D_2
				\end{bmatrix}=\begin{bmatrix}
					PS & 0 \\
					C_2S+D_2C_1 & D_2D_1
				\end{bmatrix}.
			\end{split}
		\end{equation*} 
		Similar computation implies that $T=T^*V$ if and only if 
		\begin{equation*}
			\begin{split}
				\begin{bmatrix}
					S & 0 \\
					C_1 & D_1
				\end{bmatrix}= \begin{bmatrix}
					S^* & C_1^* \\
					0 & D_1^*
				\end{bmatrix}\begin{bmatrix}
					P & 0 \\
					C_2 & D_2
				\end{bmatrix} = \begin{bmatrix}
					S^*P+C_1^*C_2 & C_1^*D_2 \\
					D_1^*C_2 & D_1^*D_2
				\end{bmatrix}.
			\end{split}
		\end{equation*}
		Again by routine calculations, we have that $V^*V=I$ if and only if
		\begin{equation*}
			\begin{split}
				\begin{bmatrix}
					P^* & C_2^* \\
					0 & D_2^*
				\end{bmatrix}\begin{bmatrix}
					P & 0 \\
					C_2 & D_2
				\end{bmatrix}=\begin{bmatrix}
					P^*P+C_2^*C_2 & C_2^*D_2 \\
					D_2^*C_2 & D_2^*D_2
				\end{bmatrix}=\begin{bmatrix}
					I & 0 \\
					0 & I
				\end{bmatrix}.
			\end{split}
		\end{equation*}
Theorem 2.6 in \cite{AglerII16} states that a commuting pair of operators $(T,V)$ is a $\Gamma$-isometry if and only if $T=T^*V$, $\|T\|\leq 2$ and $V$ is an isometry. Thus, combining the equality of the block matrices in all three cases we can say that $(T,V)$ is a $\Gamma$-isometry if and only if $\|T\| \leq 2$ and the following hold: 
		
		\begin{minipage}[t]{0.3\textwidth}
			\smallskip
			\begin{equation*}
				\begin{split}
					& 1. \ \ C_1P+D_1C_2=C_2S+D_2C_1, \\
					& 4.  \ \ C_1^*D_2=0 , \\
					& 7. \ \ C_2^*C_2=D_P^2,\\
				\end{split}
			\end{equation*}
		\end{minipage}
		\begin{minipage}[t]{0.3\textwidth}
			\smallskip
			\begin{equation*}
				\begin{split}
					& 2. \ \ D_1D_2=D_2D_1, \\
					& 5. \ \ C_1=D_1^*C_2, \\
					& 7.  \ \  C_2^*D_2=0, \\
				\end{split}
			\end{equation*}
		\end{minipage}
		\begin{minipage}[t]{0.3\textwidth}
			\smallskip
			\begin{equation}\label{eqn1003}
				\begin{split}
					& 3. \ \ S-S^*P=C_1^*C_2, \\
					& 6. \ \ D_1=D_1^*D_2, \\
					& 9. \ \ D_2^*D_2=I.
				\end{split}
			\end{equation}
		\end{minipage}
		
		Since $(T, V)$ is a $\Gamma$-isometry and $(D_1, D_2)=(T|_\mathcal{\HS^\perp}, V|_{\HS^\perp})$, it follows that $(D_1, D_2)$ is a $\Gamma$-isometry too. The operator equations in (\ref{eqn1002}) follow from (\ref{eqn1003}).
		
		\smallskip
		
\noindent $(2)\Rightarrow (1)$.	Let us assume that there is a Hilbert space $\mathcal{K}$ containing $\HS$, a $\Gamma$-distinguished $\Gamma$-isometry $(D_1, D_2)$ on $\HS^{\perp}=\mathcal{K}\ominus \HS$ and $C_1, C_2 \in \mathcal{B}(\HS, \HS^\perp)$ such that the operator identities in (\ref{eqn1002}) hold. Set 
		\[
		T=\begin{bmatrix}
			S & 0 \\
			C_1 & D_1
		\end{bmatrix} \quad \text{and} \quad 	V=\begin{bmatrix}
			P & 0 \\
			C_2 & D_2
		\end{bmatrix}
		\]
		on $\mathcal{K}=\HS \oplus \HS^\perp$.  We have by Lemma 1 in \cite{Hong} that $\sigma(T) \subseteq \sigma(S) \cup \sigma(D_1)$ and thus $r(T) \leq \max\{r(S), r(D_1)\}$. Since $(S, P)$ and $(D_1, D_2)$ are $\Gamma$-contractions, we must have $r(S) \leq 2$ and $r(D_1) \leq 2$ by Theorem 4.4 in \cite{Pal8}. Thus $r(T) \leq 2$. Theorem 2.14 in \cite{Pal8} states that $(T, V)$ is a $\Gamma$-isometry if and only if $r(T) \leq 2$ and the following hold: $TV=VT$, $T=T^*V$ and $V^*V=I$. With the given hypothesis that $(D_1, D_2)$ is a $\Gamma$-isometry and (\ref{eqn1002}) holds, we have that (\ref{eqn1003}) holds if and only if $C_1^*D_2=0$. In other words, $(T, V)$ is a $\Gamma$-isometry if and only if $C_1^*D_2=0$. By part-$(v)$ of (\ref{eqn1002}) and part-$(iii)$ of (\ref{eqn1002}), it follows that $C_1^*D_2 = C_2^*D_1D_2 =  C_2^*D_2D_1  = 0 $. Consequently, $(T, V)$ is a $\Gamma$-isometry. It is evident that $S^*=T^*|_\HS, P^*=V^*|_\HS$ and so, $(T, V)$ is a $\Gamma$-isometric dilation of $(S, P)$. By Lemma \ref{lem1014}, $(T, V)$ is $\Gamma$-distinguished.
		
		\medskip
		
		\noindent $(1) \Leftrightarrow (3)$. This part follows from Theorems \ref{Main} \& \ref{G_isometries}.		
		
		\medskip	
		
		 Moreover, if $f$ and $g$ are $\Gamma$-distinguished polynomials that annihilate $(S, P)$ and $(D_1, D_2)$ respectively, then $fg$ annihilates $(T, V)$. So, looking at the argument in the proof of $(2)\Rightarrow (1)$ it follows from Theorem \ref{Main} that $Z(fg) \cap \Gamma$ is a complete spectral set for $(S, P)$. This completes the proof.\qed
	
\medskip	
	
Now we investigate a few special cases when a $\Gamma$-distinguished $\Gamma$-contraction $(S,P)$ dilates to a $\Gamma$-distinguished $\Gamma$-unitary, especially when $\mathcal D_P$ or $\mathcal D_{P^*}$ is finite dimensional. In this connection, let us state a result due to Pal and Shalit \cite{PalShalit1} which will be useful.
		
	\begin{thm}[\cite{PalShalit1}, Theorem 4.5]\label{PalShalit4.5}
		Let $(S,P)$ be a $\Gamma$-contraction on a Hilbert space $\mathcal{H}$ such that $(S^*, P^*)$ is pure, and suppose that dim $\mathcal{D}_{P}< \infty$. Let $F$ be the fundamental operator of $(S,P)$ and let
		$
		\Lambda=\{(z_1, z_2) \in \Gamma: det(F^*+z_2F-z_1I)=0\}$. Then for every matrix-valued polynomial $f$,
		\[
		\|f(S,P)\| \leq \max\{|f(z_1, z_2)|: (z_1, z_2) \in \Lambda \cap b\Gamma\}.
		\]
		Moreover, if $\omega(F)<1$, then $\Lambda \cap \mathbb{G}_2$ is a distinguished variety in $\mathbb{G}_2$.
	\end{thm}
	
	\begin{lem}\label{5.3} 
		Let $(S,P)$ be a commuting pair of operators acting on a Hilbert space $\mathcal{H}$. Then $(S,P)$ is distinguished or $\Gamma$-distinguished if and only if $(S^*,P^*)$ is distinguished or $\Gamma$-distinguished respectively.
	\end{lem}
	\begin{proof}
		Let $q(z_1, z_2)=\overset{n}{\underset{i=0}{\sum}}\overset{m}{\underset{j=0}{\sum}}a_{ij}z_1^iz_2^j$. Consider the polynomial $\widetilde{q}(z_1, z_2)=\overset{n}{\underset{i=0}{\sum}}\overset{m}{\underset{j=0}{\sum}}\overline{a_{ij}}z_1^iz_2^j$. Clearly, $Z(\widetilde{q})=\{(z_1, z_2) : \ (\overline{z}_1, \overline{z}_2) \in Z(q)\}$. Hence, $\widetilde{q}$ is distinguished or $\Gamma$-distinguished if and only if $q$ is distinguished or $\Gamma$-distinguished, respectively. Also, $q(S,P)=0$ if and only if $\widetilde{q}(S^*, P^*)=0$.
	\end{proof}
	
		\begin{lem}\label{lem8.7}
		Let $(S, P)$ be a $\Gamma$-contraction on a Hilbert space $\mathcal{H}$. Then $(S, P)$ dilates to $\Gamma$-distinguished $\Gamma$-unitary if and only if $(S^*, P^*)$ dilates to a $\Gamma$-distinguished $\Gamma$-unitary.
	\end{lem}
	
	\begin{proof}
		It suffices to prove the necessary part. Assume that $(S, P)$ has a dilation to a $\Gamma$-distinguished $\Gamma$-unitary $(U_1, U_2)$ on $\mathcal{K}$ containing $\mathcal{H}$. Then the $\Gamma$-unitary $(U_1^*, U_2^*)$ dilates $(S^*, P^*)$ and  it follows from Lemma \ref{5.3} that $(U_1^*, U_2^*)$ is $\Gamma$-distinguished. The proof is complete.
	\end{proof}
	%-------------------------------------------------------
	\begin{thm}\label{thm8.8}
		Let $(S,P)$ be a $\Gamma$-contraction on a Hilbert space $\mathcal{H}$ such that $(S^*, P^*)$ is pure, and suppose that dim $\mathcal{D}_{P}< \infty$. Let $F$ be the fundamental operator of $(S,P)$ such that $\omega(F)<1$. Then $(S, P)$ dilates to a $\Gamma$-distinguished $\Gamma$-unitary.
	\end{thm}
	%------------------------------------------------------------------------------------
	\begin{proof}
		Applying Theorem \ref{thm:2.12} to $(S^*, P^*)$, we have that $\mathcal{H} \subseteq H^2(\mathcal{D}_P)$ and that
		$
		S=T_\phi^*|_\mathcal{H}$ and $P=T_z^*|_\mathcal{H}$, where $\phi(z)=F^*+Fz$ and $F$ is the fundamental operator of $(S, P)$. Moreover, $F$ is a matrix since $\mathcal{D}_P$ is finite-dimensional. Note that the pair $(T_\phi, T_z)$ acting on $H^2(\mathcal{D}_P)$ is a $\Gamma$-isometric dilation of $(S^*, P^*)$. We show that $(T_\phi, T_z)$ is $\Gamma$-distinguished. Define 
		\[
		\Lambda :=\{(z_1, z_2) \in \Gamma: det(F^*+z_2F-z_1I)=0\}.
		\]	
		Let $f$ be any matrix-valued polynomial. Then it follows from the proof of Theorem 4.5 in \cite{PalShalit1} that 
		\begin{equation}\label{eqn8.1}
			\|f(T_\phi^*, T_z^*)\| \leq \max\{|f(z_1, z_2)|: (z_1, z_2) \in \Lambda \cap b\Gamma\}.
		\end{equation}
		As $\omega(F)<1$, Theorem \ref{PalShalit4.5} implies that  
		$
		q(z_1, z_2)=det(F^*+z_2F-z_1I)
		$
		is a $\Gamma$-distinguished polynomial.	By (\ref{eqn8.1}), we have 
		$
		\|q(T_\phi^*, T_z^*)\| \leq \max\left\{|q(z_1, z_2)|: (z_1, z_2) \in \Lambda \cap b\Gamma\right\}=0$.	Hence, $(T_\phi^*, T_z^*)$ is $\Gamma$-distinguished. By Lemma \ref{5.3}, $(T_\phi, T_z)$ is $\Gamma$-distinguished too. Therefore, $(S^*, P^*)$ dilates to the $\Gamma$-distinguished $\Gamma$-isometry $(T_\phi, T_z)$. By Proposition \ref{prop8.1}, we have that $(T_\phi, T_z)$ has an extension to a $\Gamma$-distinguished $\Gamma$-unitary. Putting everything together, we have that $(S^*, P^*)$ has a $\Gamma$-distinguished $\Gamma$-unitary dilation. The desired conclusion follows from Lemma \ref{lem8.7}.
	\end{proof}
	
	The following two corollaries are consequences of Lemma \ref{lem8.7} and Theorem \ref{thm8.8}.
	
		\begin{cor}
		Let $(S,P)$ be a $\Gamma$-contraction on a Hilbert space $\mathcal{H}$ such that $(S, P)$ is pure, and suppose that dim $\mathcal{D}_{P^*}< \infty$. Let $F_*$ be the fundamental operator of $(S^*,P^*)$ such that $\omega(F_*)<1$. Then $(S, P)$ dilates to a $\Gamma$-distinguished $\Gamma$-unitary.
	\end{cor}
		
	\begin{cor}
		Every strict $\Gamma$-contraction on a finite-dimensional Hilbert space admits a $\Gamma$-distinguished $\Gamma$-unitary dilation.
	\end{cor}
	
	\begin{proof}
		Let $(S, P)$ be a strict $\Gamma$-contraction on a finite-dimensional space $\mathcal{H}$ and let $F$ be the fundamental operator of $(S, P)$. By Lemma 4.2 in \cite{PalShalit1} and Proposition 4.3 \cite{PalShalit1}, we have that $(S^*, P^*)$ is a pure $\Gamma$-contraction and $\omega(F)<1$. It follows from Theorem \ref{thm8.8} that $(S, P)$ dilates to a $\Gamma$-distinguished $\Gamma$-unitary.		
	\end{proof}
	We conclude this article here. A sequel of this article dealing more with minimal dilation of toral contractions and $\Gamma$-distinguished $\Gamma$-contractions will appear soon.

	\end{document}